\documentclass[11pt]{article}

\usepackage[top=1in,bottom=1.2in,left=1.25in,right=1.25in]{geometry}

\usepackage{graphicx}
\usepackage{ragged2e}
\usepackage{tabto}
\usepackage{amssymb}
\usepackage{float}
\usepackage{multirow}
\usepackage{amsmath}
\usepackage{tikz}
\usepackage{amsthm}
\usepackage{color}
\usepackage{url}
\usepackage[noabbrev,capitalise]{cleveref}
\usepackage{afterpage}
\usepackage{multicol}
\usepackage[font=footnotesize]{caption}
\usepackage[font=footnotesize]{subcaption}
\usepackage{marginnote}
\usepackage{tcolorbox}
\usepackage{enuitem}

\usepackage{newtxtext}
\usepackage{newtxmath}

\usepackage[absolute,overlay]{textpos}

\usepackage[sf,bf,medium]{titlesec}

\usepackage{bm}
\usepackage{dsfont}
\usepackage{booktabs}
\usepackage{siunitx}
\usepackage[sort,compress]{cite}

\usepackage[latin1]{inputenc}
\usepackage{caption} 
\graphicspath{{./}}

\usepackage{array,booktabs}
\usepackage{siunitx}
\newtheorem{theorem}{\sffamily Theorem}
\newtheorem{remark}{\sffamily Remark}
\newtheorem{definition}{\sffamily Definition}
\newtheorem{lemma}{\sffamily Lemma}
\newtheorem{proposition}{\sffamily Proposition}
\newtheorem{corollary}{\sffamily Corollary}

\newcommand{\pa}{\partial}
\newcommand{\be}{\begin{equation}}
\newcommand{\ee}{\end{equation}}
\newcommand{\ba}{\begin{aligned}}
\newcommand{\ea}{\end{aligned}}
\newcommand{\bea}{\begin{eqnarray}}
\newcommand{\eea}{\end{eqnarray}}

\newcommand{\Gammac}{{\Gamma_\mathbb{C}}}
\newcommand{\tGamma}{\tilde{\Gamma}}
\newcommand{\tGammat}{\tilde{\Gamma}_{\textrm{trunc}}}
\newcommand{\Ltrunc}{L_{\textrm{trunc}}}
\newcommand{\hGamma}{\widehat{\Gamma}}
\newcommand{\tgamma}{\tilde{\gamma}}

\newcommand{\erfc}{\textrm{erfc}}
\newcommand{\eps}{\varepsilon}

\newcommand{\sgn}{\operatorname{sgn}}
\newcommand{\IInt}{\operatorname{int}}
\newcommand{\tsigma}{\tilde{\sigma}}

\newcommand{\bx}{\boldsymbol{x}}
\newcommand{\by}{\boldsymbol{y}}
\newcommand{\bn}{\boldsymbol{n}}
\newcommand{\cG}{\mathcal{G}}
\newcommand{\bgamma}{\boldsymbol{\gamma}}
\newcommand{\tbgamma}{\tilde{\boldsymbol{\gamma}}}
\newcommand{\uin}{{u^{\mathrm{in}}}}
\newcommand{\kfunc}{{g}}
\newcommand{\cD}{\mathcal{D}}

\newcommand{\bbC}{\mathbb{C}}
\newcommand{\cspace}{\mathcal{C}}

\begin{document}

\begin{titlepage}

  \raggedleft
  {\sffamily \bfseries STATUS: arXiv pre-print}
  
  \hrulefill

  \raggedright
  \begin{textblock*}{\linewidth}(1.25in,2in) 
    {\LARGE \sffamily \bfseries Coordinate complexification for the Helmholtz equation with Dirichlet boundary conditions in a perturbed half-space}
  \end{textblock*}

  \normalsize

  \vspace{2in}
  Charles L. Epstein\\
   \emph{\small Center for Computational Mathematics, Flatiron Institute\\
    New York, NY 10010}\\
  \texttt{\small cepstein@flatironinstitute.org}
 
   \vspace{\baselineskip}
Leslie Greengard\\
   \emph{\small Courant Institute, NYU\\ New York, NY 10012\\and\\Center for Computational Mathematics, Flatiron Institute\\
    New York, NY 10010}\\
  \texttt{\small greengard@courant.nyu.edu} 
  
   \vspace{\baselineskip}
   Jeremy Hoskins\\
     \emph{\small Department of Statistics and CCAM, University of Chicago\\Chicago, IL 60637}\\
  \texttt{\small jeremyhoskins@uchicago.edu} 
 
   \vspace{\baselineskip}
  Shidong Jiang\\
   \emph{\small Center for Computational Mathematics, Flatiron Institute\\
    New York, NY 10010}\\
  \texttt{\small sjiang@flatironinstitute.org}
  
  \vspace{\baselineskip}
  Manas Rachh\\
  \emph{\small Center for Computational Mathematics, Flatiron Institute\\
    New York, NY 10010}\\
  \texttt{\small mrachh@flatironinstitute.org}

  \begin{textblock*}{\linewidth}(1.25in,8.5in) 
    \today
  \end{textblock*}

\end{titlepage}

\begin{abstract}
We present a new complexification scheme based on the classical double layer potential for the solution of the  Helmholtz equation with Dirichlet boundary conditions in compactly perturbed half-spaces in two and three dimensions.   The kernel for the double layer potential is the normal derivative of the  free-space Green's function, which has a well-known analytic continuation into the complex plane as a function of both target and source locations.  Here, we prove that -  when the incident data are analytic and satisfy a precise asymptotic estimate - the solution to the boundary integral equation itself  admits an analytic continuation into specific regions of the complex plane, and satisfies a related asymptotic estimate (this class of data includes both plane waves and the field induced by point sources). We then show that, with a carefully chosen contour deformation,  the oscillatory integrals are converted to exponentially decaying integrals,  effectively reducing the infinite domain to a domain of finite size.  Our scheme is different from existing methods that use complex coordinate transformations, such as perfectly matched  layers, or absorbing regions, such as the gradual complexification of  the governing wavenumber. More precisely, in our method, we are still solving a boundary integral equation, albeit on a truncated, complexified version of the original  boundary. In other words, no volumetric/domain modifications are introduced.  The scheme can be extended to other boundary conditions, to open wave guides and to layered media. We illustrate the performance of the scheme with  two and three dimensional examples.
\vspace*{1ex}


\end{abstract}

\pagestyle{myheadings}
\thispagestyle{plain}
\markboth{C. Epstein, L. Greengard, J. Hoskins, S. Jiang, and M. Rachh}
{Coordinate Complexification for the Helmholtz Dirichlet problem in a perturbed half-space}

\section{Introduction} \label{introduction}

In this paper we consider the classical problem of sound soft time-harmonic wave scattering from an infinite boundary. 
As a model, consider the region to be the compactly perturbed upper half space in two dimensions, denoted by $\Omega$ with boundary $\Gamma$. The scattered field $u$ satisfies,
\begin{equation}
\label{eq:bvp}
\begin{aligned}
(\Delta + k^2)u(\bx) & = 0 \, , & \bx \in \Omega, \\
u(\bx) &=f(\bx) \, ,&  \bx \textrm{ on } \Gamma, \\
\lim_{|\bx| \to \infty} |\bx|^{1/2} (\pa_{|\bx|} -ik)u(\bx) &=0 \, ,& \textrm{uniformly in angle},
\end{aligned}
\end{equation}
where $f$ satisfies suitable conditions at infinity. Such problems arise frequently in modeling acoustic and electromagnetic scattering from large obstacles and layered media.  

For scattering from compact obstacles, boundary integral equations (BIEs) are a robust and mature class of techniques. Rather than discretizing the partial differential equation (PDE) directly, the scattered field is represented using single and double layer potentials with unknown densities. In addition to reducing the dimension of the unknowns, this reduction typically leads to discretized linear systems which are as well-conditioned as the underlying physical problem. Unlike PDE-based methods however, the discretized BIEs give rise to dense matrices. For large systems this necessitates the use of fast algorithms either for applying these matrices quickly (see for example~\cite{brandt1990multilevel,brandt1998multilevel,bruno1, bruno2,carrier1988sisc,crutchfield2006remarks,darden1993particle,eastwood1974shaping,fong2009black,greengard2023dual,greengard1987jcp, hockney2021computer,malhotra2015pvfmm,rokhlin1990rapid,toukmaji1996ewald}) or to construct a compressed representation of the inverse (see for example~\cite{gillman-dir_hss-2012,ho-2012,martinsson-2005,kong-hodlrdir-2011,Mar-hss-2011,hackbusch-h-1999,hackbusch-h-2000,Borm-h-2003,ambikasaran-2014,Borm-h2-2010,hackbusch-h2-2000,
  Jiao-h2-ce-2017,minden-str_scel-2017,coulier-ifmm_prec-2017,Beb-hlu-2005,borm-h2lu-2013,Hakb-h2-lib,minden-str_scel-2017,jiang2022skel,sushnikova-ce-2018,sushnikova2023fmm}). These techniques apply, for example, to the Laplace, Stokes, Helmholtz, Maxwell, Dirac, and elastic wave equations, subject to a variety of boundary conditions, e.g. Dirichlet, Neumann, Robin, transmission, traction, and mixed, depending on the context. Moreover, there are a number of modifications of these methods that permit the treatment of domains with corners, edges, and multiple junctions (see for example~\cite{helsing2013solving,serkh2016solution,ginn2023lightning,hoskins2019numerical,hoskins2020solution,bremer2012nystrom}).

The application of BIEs to non-compact scatterers, such as half-plane problems and infinite interfaces, poses significant challenges both algorithmically and analytically. In such cases the data and/or the solution of the integral equation can decay slowly and even the invertibility of the integral equation can be a subtle question. Over the past several decades, a number of methods have been developed to circumvent these issues.
One of the standard approaches is to discretize the problem partially in the Fourier domain using {\it Sommerfeld integrals}~\cite{sommerfeld1909,van1935,weyl1919,cai2013computational,okhmatovski2004evaluation,paulus2000accurate,perez2014high,ochmann2004complex,koh2006exact,o2014efficient}. These methods tend to be inefficient in the presence of scatterers or point sources close to the infinite interfaces, which lead to slow decay in the Sommerfeld representation of the solution. These issues have been effectively addressed using {\it windowed Green's function} approaches, which use partitions of unity to discretize the perturbation in physical space and the infinite flat part of the geometry in Fourier space~\cite{bruno2016windowed,lai2018new}. 
While effective in practice, these methods don't directly address the solvability of either
the half-space integral equation or its truncated counterpart.
After discretization, there have also been advances in the design of fast algorithms for the solution of the corresponding linear systems  \cite{chohuangchencai}.

Half-space matching methods are another class of methods that have been shown to be effective for solving scattering problems on infinite interfaces with dissipative media ($\Im (k) > 0$)~\cite{dhia2018halfspace}. The idea is to represent the solution via their unknown traces on boundaries of intersecting half-spaces (the interior of which encloses the perturbation of the boundary), and use the dissipation to effectively truncate the integral equations. This work was extended to the case of non-dissipative media ($\Im (k) = 0$) by solving for the trace of the solution along complexified contour where the solution decays due to its outgoing nature~\cite{bonnet2022complex,bonnet2022half}. Extensions of this approach, and the analysis therein, can be particularly challenging for more complicated settings such as transmission problems, and layered media. In~\cite{Bonnet2024water_waves} a complexification method is used, in the context of linear water waves, to solve the Laplace equation with an unusual radiation condition in a strip in $\mathbb{R}^2$.

A related set of methods are perfectly matched layers (PMLs), which are often used for truncating unbounded regions for PDE discretizations of the problem. These methods use the outgoing nature of the solutions to extend them to a two dimensional manifold in $\mathbb{C}^{2}$ where the solution is exponentially decaying and can be effectively truncated~\cite{berenger1994perfectly,chew19943d,dyatlov2019mathematical}. 

In this paper, we take a different approach, `complexifying' the coordinates of the interface itself. This is closely related to the work of \cite{lu2018perfectly}, which the authors refer to as a PML-based boundary integral method. 
The main contribution of this paper is a proof of invertibility of the integral equation in suitable spaces which encode the outgoing nature of the data and solutions. 
We show that the complexified integral equation is equivalent to the original
equation and that its truncation introduces an exponentially small error. The 
approach is simple to implement, involving a small and straightforward modification of existing methods, and
extends naturally to transmission problems, layered media, waveguides, leaky quantum graphs, and certain models of edge states in insulating media.

\section{Preliminaries}
\label{sec:prelim}
Let us first consider half-space problems in which the boundary is a compact perturbation of the $x_{1}$-axis ($\{(x_1,x_2)\in\mathbb{R}^2\,|\, x_2 = 0\}$). In particular, we suppose the boundary $\Gamma$ is the $x_1$-axis in $\mathbb{R}^{2}$, except for the interval
$x_{1} \in [-L+\delta, L -\delta]$ for some $\delta > 0$, where $\Gamma$ is
parameterized by $\bgamma(t) : \mathbb{R} \to \mathbb{R}^2$, with $|\gamma_{1}(s)| \leq L-\delta$ for all $|s|\leq L-\delta$, and $\bgamma(s) = (s,0)$ for all $|s| \geq L-\delta$. Let $\Omega$ denote the infinite region corresponding to the upper half space whose boundary is $\Gamma$, see~\cref{fig:boundary}. Let $\bn$ denote the positively oriented normal to the boundary $\bgamma$, which for $x_{1} \in \mathbb{R} \setminus [-L+\delta,L-\delta]$ is given by 
$\bn = (0,-1)$. 
\begin{figure}[!ht]
\begin{center}
\includegraphics[width=0.8\linewidth]{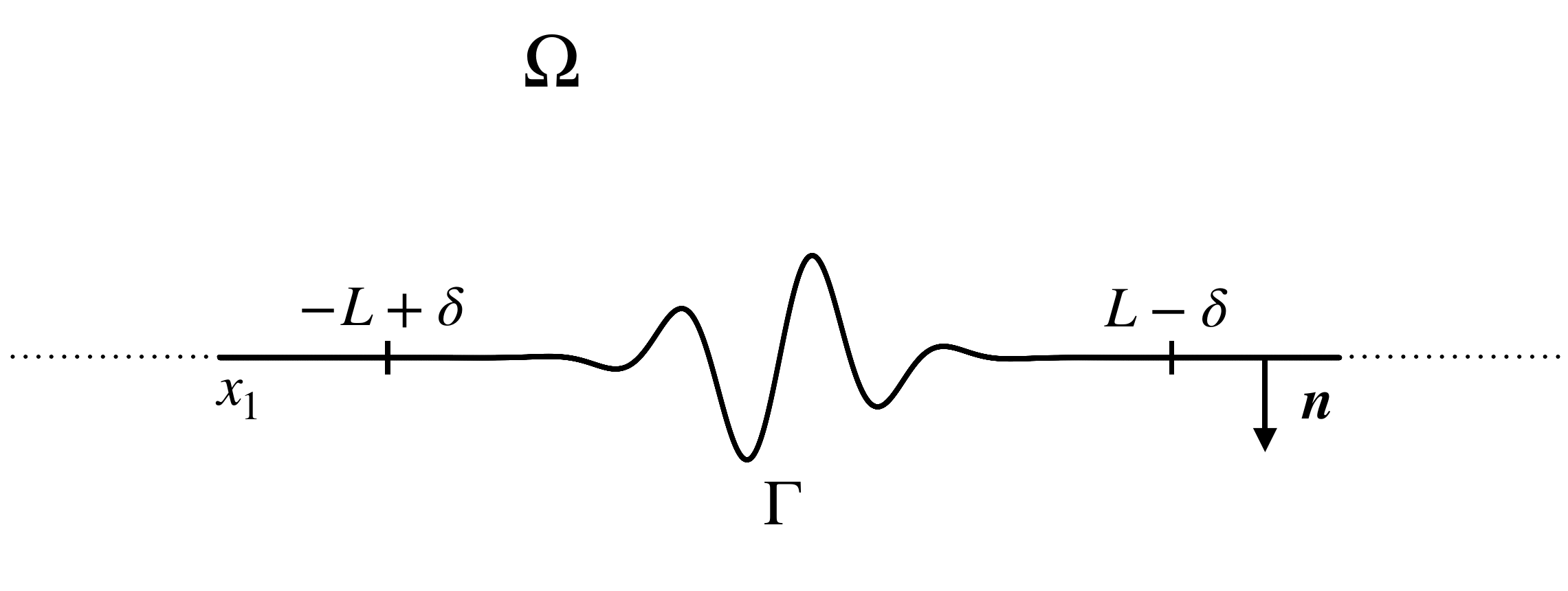}
\caption{Perturbed half-space geometry}\label{fig:boundary}
\end{center}
\end{figure}

\subsection{Boundary integral equations for Helmholtz  Dirichlet problems}
A standard approach to solve~\cref{eq:bvp} is to represent the potential $u$ via a {\it Helmholtz double layer potential} denoted by $\cD_{\Gamma}[\sigma]$, and given by,
\begin{equation}
u(\bx) = \cD_{\Gamma}[\sigma](\bx) = \frac{ik}{4} \int_{\Gamma} \frac{H_{1}^{(1)}(k|\bx - \by|) ((\bx - \by) \cdot \bn(\by))}{|\bx- \by|}\, \sigma(\by)\, {\rm d}s_{\by} \, ,
\end{equation}
where $H_{1}^{(1)}(z)=-\pa_z H^{(0)}_1(z)$ denotes the Hankel function of the first kind and order $1$. 
For clarity of notation in the following we omit the superscripts on the Hankel functions.

Upon imposing the boundary condition, and using the standard jump relations for the Helmholtz double layer potential~\cite{colton}, we obtain the following integral equation for the density $\sigma$, 
\begin{equation}\label{eqn:first_BIE}
-\frac{1}{2} \sigma(\bx) + D_{\Gamma}[\sigma](\bx) = f(\bx) \, , \quad  \bx \in \Gamma \,  ,
\end{equation}
where $D_{\Gamma}[\sigma]$ is the principal value part of the double layer potential restricted to the boundary $\Gamma$.

\section{Intuition for coordinate complexification}

In this section we give an informal description of the coordinate complexification method. We first temporarily set aside solvability of the boundary integral equation~\eqref{eqn:first_BIE}, supposing for now that a solution $\sigma$ exists and has sufficient algebraic decay at infinity.
Rearranging the terms in~\eqref{eqn:first_BIE}, we see that
\begin{equation}
\sigma(\bx) = -2f(\bx) + 2D_{\Gamma}[\sigma](\bx) \, , \quad  \bx \in \Gamma \, .
\end{equation}
The key idea of this work is to observe that if $f$ admits an analytic extension to the upper half plane for $x_{1}>L$, and lower half plane with $x_{1}<-L$, and satisfies the estimate
\begin{equation}
f((x_{1} + \textrm{sgn}(x_{1}) \eta,0)) \lesssim e^{-k\eta}/(1+|x_{1}|)^{1/2} \, ,\quad |x_{1}|>L, \eta>0 \,,
\end{equation}
then the density $\sigma$ also admits an
analytic extension to the upper half plane for $x_{1}>L$, and lower half plane for $x_{1}<-L$ and satisfies a similar estimate. This is evident from the fact that $D_{\Gamma}[\sigma](\bx)$ can be analytically continued to $(x_{1}+i\textrm{sgn}(x_{1})\eta,0)$ for $\eta>0$, and $|x_{1}|>L$ with its kernel satisfying the estimate
\begin{equation}
\label{eq:kern-decay}
\kfunc(x_{1}+i\textrm{sgn}(x_{1}) \eta, s)\lesssim e^{-k\eta}/(1+|x_{1}|)^{1/2} \, .
\end{equation}

The choice of extension of the solution, and the behavior of the kernel in those regions of the complex plane follows from the fact that the scattered field is outgoing, and hence both the Green's function $\Phi$ of the Helmholtz equation, and the scattered field behave like 
\begin{equation}
u, \Phi \sim \frac{e^{ik |x_{1}|}}{\sqrt{|x_{1}|}} \, ,\quad \text{as} \, x_{1} \to \pm \infty \, .
\end{equation}
Note that, for $x_1\in [0,\infty),$ $|x_1|=\sqrt{x_1^2},$  and for $x_1t\geq 0,$
\begin{equation}
    \Im\sqrt{(x_1+it)^2}= \Im\sqrt{x^2_1-t^2+2ix_1t}\geq 0.
\end{equation}
Therefore
it  follows that  that the analytic continuations both $u$ and $\Phi$ decay exponentially as in~\cref{eq:kern-decay} in the upper half plane for positive $x_{1}$, and the lower half plane for negative $x_{1}$, with $|x_{1}|$ large enough in either case.  Here and throughout this paper $\sqrt{z}$ is defined in $\bbC\setminus (-\infty,0)$ with $\sqrt{x}>0$ for $x\in (0,\infty).$

It is natural to ask what sort of boundary data, $f,$ has such an extension, since the boundary data itself may not be outgoing as is the case with plane waves. It turns out that, after subtracting the contribution from Snell's law, the remaining boundary data is compactly supported on the perturbed part of the boundary, and hence admits an exponentially decaying extension to the required regions of the complex plane. Moreover, data generated by point or compactly supported volumetric sources in $\mathbb{R}^{2}$, which are bounded away from the interface, also naturally satisfy these estimates. These boundary conditions encompass a large set of relevant boundary value problems of interest.

Now since the density $\sigma$ formally admits an analytic extension to the upper half plane for $x_{1}>L$, and lower half plane for $x_{1}<-L$, it is tempting to deform the integral equation itself to complex coordinates and solve for the density for a different curve in the complex plane, rather than the real axis. This is advantageous since on the real axis the density admits at most algebraic decay of the form $(1+|x_{1}|)^{-1/2}$, which needs to be dealt with in some manner in order to obtain high accuracy solutions efficiently, i.e. without having to discretize the problem out to large enough $R\lesssim 1/\varepsilon^{2}$. In fact, ideally, it is desirable to have $O(1)$ unknowns describing the solution for $|x_{1}|\geq L$, since the interface is flat, it should be possible to exploit the analytic Green's function in this setting.  

Suppose that the density and the integral equation can be continued to a curve of the form $x_{1}+i \tau(x_{1})$, with the property that $\tau(x_{1})$ is strictly monotonically increasing, and $\tau'(x_{1})>c>0$ for $|x_{1}|>L$, with $\tau(x_{1}) = 0$ for $|x_{1}|\leq L$, then the density along this contour would be exponentially decaying as $|x_{1}| \to \infty$, thereby allowing for its efficient discretization for $|x_{1}|>L$. Let $\tGamma$ be the curve which agrees with $\Gamma$ for $|x_1|\leq L$, and is given by $(x_{1}+i\tau(x_{1}),0)$, for $|x_{1}|>L$, then the integral equation on the deformed contour is given by
\begin{equation}
-\frac{1}{2} \sigma(\bx) + \mathcal{D}_{\tGamma} [\sigma](\bx) = f(\bx) \,, \quad  \bx \in \tGamma \, .
\end{equation}
In this setting, since the density would now asymptotically behave as
\begin{equation}
\label{eq:sig_asymp}
\sigma((x_{1} + i\tau(x_{1}),0)) \sim e^{-k |\tau(x_{1})|}/(1+|x_{1}|)^{1/2} \, , \quad |x_{1}| \to \infty \, ,
\end{equation}
we are free to choose an appropriate contour $\tau(x_{1})$, and a discretization adapted to the exponential decay of the density such that it could be represented using $O(\log(1/\varepsilon))$ points on this part of the contour. 
The rest of the paper provides a rigorous analysis of this formal picture.

\section{Main results}\label{sec:main_res}

We require our data and densities to satisfy suitable analyticity and decay conditions. The following subsets of $\mathbb{C} \times \mathbb{R}$ play an essential role in defining these conditions, and are closely related to the notion of outgoing solutions to the Helmholtz problem.
\begin{definition}\label{def:gammas}
Given $\Gamma \subset \mathbb{R}^2$ and $L\in \mathbb{R}_+,$ as in \cref{fig:boundary}, we define the sets $\Gamma_{L}, \Gamma_{M}, \Gamma_R,$ and $\Gammac$ in $\mathbb{C}\times \mathbb{R}$ by
\begin{align*}
    \Gamma_{L}&:= \left\{(z,0) \in \mathbb{C}\times\mathbb{R}\,|\, \Re{z}< -L,\, \Im{z}  \le 0\right\},\\
        \Gamma_{M}&:= \left\{(z,y) \in \mathbb{C}\times\mathbb{R}\,|\, \Im{z}=0, \text{ and }(z,y) = (x,y)\in \Gamma,\, |z|\le L\right\},\\
            \Gamma_{R}&:= \left\{(z,0) \in \mathbb{C}\times\mathbb{R}\,|\, \Re{z}> L,\, \Im{z}  \ge 0\right\},
\end{align*}
and
$$\Gammac := \Gamma_L \cup \Gamma_M \cup \Gamma_R.$$
\end{definition}

It is also be convenient to define certain families of curves in $\mathbb{C}\times \mathbb{R}.$
\begin{definition}
Let $\mathcal{G}$ denote the set of all simple $C^\infty$ curves in $\Gammac$ satisfying the following conditions:
\begin{enumerate}[label=(\roman*)]
    \item If $\tGamma \in \mathcal{G},$ then $\tGamma$ goes to infinity in $\Gamma_L$ and $\Gamma_R,$ i.e. there exists a parameterization $\tbgamma$ of $\tGamma$ such that $\tbgamma(t) \subset \Gamma_L = (\tgamma_{1}(t),0)$ for $t$ sufficiently small and $\lim_{t \to -\infty} |\tgamma_1(t)|^2 = \infty.$ Similarly, $\tbgamma(t) \subset \Gamma_R = (\tgamma_{1}(t),0)$ for $t$ sufficiently large and $\lim_{t \to \infty} |\tgamma_1(t)|^2 = \infty.$ 
    \item If $\tGamma \in \mathcal{G}$ then for any $\bx,\by\in \tGamma,$ whenever either $\Im{x}_1$ or $\Im{y}_1$ is non-zero,
    $${\rm sgn}\, \Re (y_1-x_1) = {\rm sgn} \,\Im (y_1-x_1),$$
    where the subscript `1' denotes the first component. In particular, the imaginary part of $\tGamma$ is monotonically increasing in $\Gamma_L$ and $\Gamma_R$ as a function of its real part. 
\end{enumerate}
Finally, for any $\bx \in \Gammac$ we let $\mathcal{G}_{\bx}$ denote the subset of curves in $\mathcal{G}$ which pass through $\bx.$
\end{definition}

\begin{remark}
    We note that $\Gamma \in \mathcal{G},$ and so this set is non-empty. It also follows straightforwardly from the definitions that $\mathcal{G}_{\bx}$ is non-empty for any $\bx \in \Gammac.$
\end{remark}
The following spaces will be used in our analysis and the statements of our main results.
\begin{definition}
For any positive real numbers $\alpha$ and $\beta$ we define the space $\cspace^{\alpha,\beta}$ to be the set of all continuous functions on the closure of $\Gammac$ which are analytic in $\IInt\Gamma_L\cup \IInt\Gamma_R$, and further satisfy the condition
\begin{equation}
\sup_{\bx \in \overline{\Gammac}} e^{\beta\, |\Im \bx|}(1+|\bx|)^{\alpha}  |f(\bx) | < \infty \, .
\end{equation}
The norm on $\mathcal{C}^{\alpha,\beta}$ is given by
$$ \|f\|_{\alpha,\beta} :=\sup_{\bx \in \overline{\Gammac}} e^{\beta\, |\Im \bx|}(1+|\bx|)^{\alpha}  |f(\bx) |.$$
\end{definition}

Using the fact that a uniformly convergent sequence of analytic functions has an analytic limit, we easily prove the following proposition:
\begin{proposition}
   For any positive real numbers $\alpha$ and $\beta$  the space $\cspace^{\alpha,\beta}$,  with the norm $\|f\|_{\alpha,\beta}$ is a Banach space.
\end{proposition}

We also require an analytic extension of the double layer kernel and operator to certain regions of the complex plane.

\begin{definition}
    Suppose that $\Gamma$ and $L$ are as defined in \cref{sec:prelim}, and $\Gamma_{L,M,R,\mathbb{C}}$ as defined in \cref{def:gammas}. We define the function $\bn_{\mathbb{C}} : \Gammac \to \mathbb{C}^2$ by 
    \begin{align}
    \bn_\mathbb{C}(\bx):= \begin{cases}
        \bn(\bx) ,\quad & \bx \in \Gamma_M,\\
        (0,-1), \quad & \bx \in \Gamma_L,\Gamma_R,
    \end{cases}
    \end{align}
    and the function $\kfunc:\mathbb{C}^{2} \times \Gammac \to \mathbb{C}$ by
    \begin{equation}
\kfunc(\bx, \by) := -\frac{ikH_{1}^{(1)}(k  r(\bx,\by)) (\bx-\by)\cdot \bn_{\mathbb{C}}(\by)}{2r(\bx,\by)},
\end{equation}
where $r = \sqrt{(x_1-y_1)^2+(x_2-y_2)^2}$ is an analytically continued distance function.
\end{definition}
\begin{remark} Where the distance function is analytically continued, the assumption that $\sgn\Re{(y_1-x_1)}=\sgn\Im{(y_1-x_1)}$ implies that that $\Im{(x_1-y_1)^2}>0,$ hence the square root is well defined and lies in the first quadrant.
\end{remark}
\begin{definition}
Let $\tGamma \in \mathcal{G}$ and let $\bgamma:\mathbb{R} \to \tGamma$ be a smooth parameterization of $\tGamma.$ We define the integral operator $K_{\Gamma_{\mathbb{C}}}$ by
$$K_{\tGamma}[\sigma](\bx) = \int_{-\infty}^\infty \,\kfunc(\bx,\bgamma(s))\,\sigma(\bgamma(s))\,\sqrt{\gamma_1'(s)^2+\gamma_2'(s)^2}\,{\rm d}s, \quad \bx \in \tGamma.$$
For ease of notation we will also write this as 
$$K_{\tGamma}[\sigma](\bx) = \int_{\tGamma}\,\kfunc(\bx,\by)\,\sigma(\by)\,{\rm d}y.$$
\end{definition}

\begin{remark}\label{rem:vanish}
It follows immediately from the definitions of $k$ and $\bn_{\mathbb{C}}$ that 
$$\kfunc(\bx,\by)\equiv 0, \,\, {\rm if}\,\, \bx,\by \in \Gamma_L \cup \Gamma_R.$$
\end{remark}

We now state the main analytical result of this paper.
\begin{theorem}
\label{thm:main}
Let $\alpha \in (0,1/2)$ and $0<\beta\leq k$. For any $f\in \mathcal{C}^{\alpha,\beta},$ there exists a unique $\sigma \in \mathcal{C}^{\alpha,\beta}$ satisfying
\begin{align}\label{eq:on_curve}
\sigma(\bx) + K_{\tGamma}[\sigma](\bx) = f(\bx), \quad {\rm for\, all\,\,}\tGamma \in \mathcal{G}_{\bx},\,{\rm and}\,\,\bx \in \tGamma,
\end{align}
and~\cref{eq:on_curve} restricted to any curve $\tGamma \in \mathcal{G}_{\bx}$ has a unique solution.

Suppose further that $f\in \cspace^{1/2,k}$ and, along the real axis,  $f$ has  asymptotic expansions
\begin{equation}
\label{eq:asymp_f}
f(x_{1}) = \frac{e^{ik|x_{1}|}}{|x_{1}|^{\frac{1}{2}}}\sum_{\ell=0}^N \frac{a_\ell^\pm}{|x_{1}|^\ell} + \mathcal{O}\left(|x_{1}|^{-N-3/2}\right), \quad {\rm as}\, \pm x_1 \to \infty ,
\end{equation}
for any $N>0$.
Then the solution $\sigma$ above is also in $\cspace^{1/2,k}$, satisfies a similar estimate along $\Gamma$, and $\mathcal{D}_{\Gamma}[\sigma](\bx)$ is a solution of~\cref{eq:bvp}.
Moreover, for any curve $\tGamma \in \mathcal{G}$, 
$\mathcal{D}_{\tGamma}[\sigma](\bx)$ agrees with $\mathcal{D}_{\Gamma}[\sigma](\bx)$ for all $\bx \in \Omega \cap \{|x_{1}| < L\}$.
\end{theorem}

This theorem plays a critical role in the construction of an efficient discretization as it asserts that the integral equation can be imposed on any contour $\tGamma \in \mathcal{G}$, and recover the restriction of the density along that contour. Moreover if $f$ is outgoing in the sense of~\cref{eq:asymp_f}, then the density is also outgoing and rapidly decaying along appropriately chosen contours, for example, ones for which the derivative of the imaginary part is bounded below by a positive constant. Along, all such contours, $\sigma$ can be represented to accuracy $\epsilon$ using $O(\log{(1/\varepsilon)})$ degrees of freedom.

\section{Overview of the proof}

The first stage of the proof is to use the operators  $K_{\tGamma}$ to define a new integral operator $K:\cspace^{\alpha,\beta} \to \cspace^{\alpha,\beta}$ for any $0<\beta\leq k$ and $\alpha \in (0,1/2),$ and prove its compactness.

We begin with the following Proposition, relating the values of $K_{\tGamma}[\sigma]$ for different $\tGamma \in \mathcal{G}_{\bx}.$
\begin{proposition}\label{prop:Kdef}
Suppose $\sigma \in \cspace^{\alpha,\beta}$ for some $0<\beta \leq k$ and $\alpha \le 1/2.$ Fix $\bx \in \Gammac.$ For each $\tGamma \subset \mathcal{G}_{\bx},$ let $K_{\tGamma}[\sigma]$ be defined as in Definition \ref{def:K_def}. Then, for any other curve $\hGamma$ in $\mathcal{G}_{\bx},$
$$K_{\tGamma}[\sigma](\bx)=K_{\hGamma}[\sigma](\bx).$$
Thus $K_{\tGamma}[\sigma](\bx)$ is well-defined as a function from $\mathcal{G}_{\bx}$ to $\mathbb{C},$ independent of the choice of $\tGamma.$ Moreover, the function so defined is in $\cspace^{\alpha,\beta}.$
\end{proposition}

The previous proposition justifies the following definition.
\begin{definition}\label{def:K_def}
Suppose  $\sigma \in \cspace^{\alpha,\beta}$ for some $0<\beta \leq k$ and $\alpha \le 1/2.$ Define the operator $K:C^{\alpha,\beta} \to C^{\alpha,\beta}$ by
$$K[\sigma](\bx) = K_{\tGamma}[\sigma](\bx),$$
for any $\bx \in \tGamma$ and any $\tGamma \in \mathcal{G}_{\bx}.$
\end{definition}

Given the above definition, it is now straightforward to show the following.
\begin{lemma}\label{lem:compact}
    Let $K$ be the integral operator defined above. For any $\alpha \in (0,1/2),$ $0<\beta\le k,$ the operator $K$ is compact from $\cspace^{\alpha,\beta}$ to itself.
\end{lemma}

The second stage is to show the invertibility of $I+K.$ The following theorem proves that the restriction of $(I+K)$ to any curve $\tGamma$ is injective.

\begin{theorem}\label{thm:inject}
   Suppose  $\sigma \in \cspace^{\alpha,\beta}$ for some $0<\beta \leq k$ and $\alpha \in (0,1/2)$ and let $K[\sigma]$ be the function given in \cref{def:K_def}. If $\sigma + K[\sigma]$ vanishes along any curve $\tGamma \in \mathcal{G}$ then $\sigma$ vanishes in all of $\Gammac.$ 
\end{theorem}

From this we obtain the following result on the invertibility of our equation.

\begin{theorem}\label{thm:invert}
    For any $f \in \cspace^{\alpha,\beta},$ $0<\alpha<1/2, 0 < \beta \leq k$, there exists a unique solution to the integral equation 
    $$\sigma(\bx) + K[\sigma](\bx) = f(\bx), \quad \bx \in \Gammac.$$
    Moreover, for any $\tGamma \in \mathcal{G},$ the integral equation
    \begin{align}\label{eq:on_curve2}
    \sigma(\bx) + K_{\tGamma}[\sigma](\bx) = f(\bx), \quad \bx \in \tGamma
    \end{align}
    has a unique solution in $\cspace^{\alpha,\beta}.$
\end{theorem}

These results are proved in the following sections.

\section{Proof of the main result}

\subsection{Preliminary results on the uniqueness of certain Helmholtz problems in halfspaces}

The following result has been known since the early 1960s.
\begin{theorem}
\label{thm:pde-uniqueness}
    Let $u \in C^2(\Omega) \cap C(\overline{\Omega})$ be a solution to the homogeneous Helmholtz equation in $\Omega,$ for which a normal derivative on $\partial \Omega$ exists. Suppose, in addition, that either $u$ or $\frac{\partial u}{\partial n}$ vanishes on $\partial \Omega,$ and that
    \begin{enumerate}
    \item $$\lim_{r\to \infty} \sqrt{r}\left[\frac{\partial}{\partial r} u(\hat{\theta} r) - ik u(\hat{\theta}r)\right] =0, $$
    \item $$\limsup_{r\to \infty} |u(\hat{\theta}r)| < \infty$$
    \end{enumerate}
    uniformly in $\hat{\theta} = (\theta_1,\theta_2) \in S^1$ with $\theta_2 >0,$ then $u \equiv 0.$
\end{theorem}
The proof can be found in~\cite{odeh1963uniqueness}, for example.
\begin{remark}
    The above theorem applies equally well to the lower half-space $\mathbb{R}^2 \setminus \Omega.$
\end{remark}

The following theorem, proved in \cite{epstein2023solving}, gives conditions on the density for which the corresponding double layer potential is outgoing.
\begin{theorem}[\cite{epstein2023solving}]
\label{thm:d_pde}
    Suppose that $\sigma \in C^1(\Gamma;\mathbb{C})$ is a density possessing an asymptotic expansion of the form
    $$\sigma({\bx}) = \frac{e^{i|\bx|k}}{|\bx|^{\frac{1}{2}}}\sum_{\ell=0}^N \frac{a_\ell^\pm}{|\bx|^\ell} + \mathcal{O}\left(|\bx|^{-N-3/2}\right), \quad {\rm as}\, \pm x_1 \to \infty.$$
    Suppose further that we define the function $u:\Omega \to \mathbb{C}$ by
$$u(\bx) = \cD_\Gamma[\sigma](\bx), \quad {\rm for}\, \bx \in \Omega.$$
Then,
\begin{enumerate}
    \item $u$ satisfies the homogeneous Helmholtz equation in $\Omega$
    \item $$\lim_{\by \to \bx \in \Gamma} u(\by) = -\frac{1}{2} \sigma(\bx) + D_\Gamma[\sigma](\bx)$$
    \item $\frac{\partial u}{\partial n}$ exists everywhere on $\Gamma$.
\end{enumerate}
    Moreover, for some $M$ depending on $N,$ $u$ has the following asymptotic expansion
    $$ u(\hat{\theta}r) = \frac{e^{ikr}}{\sqrt{r}} \sum_{j=0}^M \frac{\alpha_j(\hat{\theta})}{r^j} + \mathcal{O}(r^{-M-1}),$$
    for some coefficients $\alpha_j$ which are smooth away from $\hat{\theta} = \pm (1,0).$ Moreover, these coefficients have smooth extensions to $\hat{\theta}\to \pm(1,0)$ with error estimates which are uniformly correct.
\end{theorem}
\begin{remark}
    Intuitively, the previous theorem shows that solutions $u$ generated by an `outgoing' density are also outgoing. An identical result holds for the lower half-space $\mathbb{R}^2\setminus \Omega.$
\end{remark}

\subsection{Proof of \cref{prop:Kdef} and \cref{lem:compact}}
We now turn to the proofs that $K$ in \cref{def:K_def} is well-defined and compact. We first establish bounds on $K_{\tGamma}[\sigma]$ when $\sigma \in \cspace^{\alpha,\beta},$ $\alpha \in (0,1/2),$ $0<\beta\leq k,$ and $\tGamma \in \mathcal{G}.$

Before doing this, we first review several standard inequalities which will be useful in formulating bounds on $K_{\Gamma_{\mathbb{C}}}.$ The first lemma is elementary (the result follows from Equation 10.17.13 of \cite{nisthandbook},
for example).

\begin{lemma}
   There is a constant, $C,$ so that for any $z \in \mathbb{C},$
        $$|H^{(1)}_1(z)| \le C e^{-\Im{z}}\left( \frac{1}{\sqrt{|z|}}+\frac{1}{|z|}\right).$$
\end{lemma}

Our next lemmas are similarly elementary, but summarize several identities pertaining to complex square roots which will be used in subsequent proofs.
\begin{lemma}\label{lem:rt_bnd}
    Suppose $z$ is a complex number with  $\arg(z) \in [0,\pi/2]\cup [\pi,3\pi/2],$ and $|\Re z| > L$ for some positive constant $L.$ Moreover suppose $x,y$ are two real numbers with $x \in [-L,L]$ and $y \in [-H,H],$ where $H$ is an arbitrary positive constant.
    \begin{enumerate}
    \item If $H^2 < |z-L|^2$ then
    $$\Im \sqrt{(z-x)^2+y^2} = |\Im {z}| + R(x,y,z),$$
    where $|R(x,y,z)| \le \frac{H^2}{2|z-L|^2}.$
    \item If $H^2 < \frac{1}{2}|z-L|^2$ then 
    $$|\sqrt{(z-x)^2+y^2}|\ge {|z-L|}/{\sqrt{2}}$$
    \item If $|x|<L-\delta$ for some $\delta>0,$ then 
    $$|\sqrt{(z-x)^2+y^2}| > \sqrt{2(\sqrt{2}-1)}\, \delta.$$
    \end{enumerate}
\end{lemma}

The previous lemmas immediately imply the following bound on the kernel $\kfunc.$
\begin{lemma}\label{lem:sqrt}
     There exists a positive constants $C_1$ and $C_2$, depending on $L,$ $\Gamma,$ and $k$ such that
    \begin{enumerate}
        \item $$|\kfunc(\bx,\by)|\le C_1 \frac{e^{-k |\Im x_1|}}{\sqrt{1+|\bx|}},$$
        for all $\bx = (x_1,0) \in \overline{\Gamma_L} \cup \overline{\Gamma_R}$ and all $\by \in \Gamma_M$
         \item $$|\kfunc(\bx,\by)|\le C_2 \frac{e^{-k |\Im y_1|}}{(1+|\by|)^{\frac 32}},$$
        for all $\by = (y_1,0) \in \overline{\Gamma_L} \cup \overline{\Gamma_R}$ and all $\bx \in \Gamma_M$.
        \item For $\bx, \by \in \Gamma_M,$ $g$ is continuously differentiable.
    \end{enumerate}
    Moreover, $\kfunc$ vanishes if both $\bx,\by \in \Gamma_L \cup \Gamma_R.$
\end{lemma}
\begin{proof}
For notational clarity, here we temporarily write $\Gamma_L$ and $\Gamma_R$ in place of $\overline{\Gamma_L}$ and $\overline{\Gamma_R}.$ 

For $\bx,\by \in \Gamma_L \cup \Gamma_R,$ the kernel is zero (see \cref{rem:vanish}). For $\bx,\by \in \Gamma_M$ the result is classical, see~\cite{kress2014}, for example.

Suppose now that $\bx \in \Gamma_R$ and $\by \in \Gamma_M.$ We begin by observing that for $|\by|>L-\delta,$ $\kfunc(\bx,\by) \equiv 0$ for all $\bx \in \Gamma_R.$ Thus, it remains to consider points $\bx=(x_1,0),$ and $\by=(y_1,y_2)$ such that $\Re (x_1-y_1)>\delta.$ Using \cref{lem:rt_bnd}, together with the continuity of $\kfunc$ away from the diagonal, then shows that for any $\rho >0$ there exists a positive constant $C_\rho$ such that
$$ |\kfunc(\bx,\by)| < C_\rho,$$
for all $|\bx-(L,0)|<\rho$ and all $\by \in \Gamma_M.$ In the following we set $\rho = 2 \sup_{\by \in \Gamma_M}|y_2|=:H.$

We now bound $\kfunc$ for $|\bx-(L,0)| >\rho.$ From the previous lemma, \cref{lem:rt_bnd}, we have
 $$\Im r(\bx,\by) = |\Im {x_1}| + R(\bx,\by),$$
where $|R| \le H^2/(2|\bx-(L,0)|^2) \le 1/2.$ Moreover, again by \cref{lem:rt_bnd},
$$|r(\bx,\by)| >|\bx-(L,0)|/\sqrt{2}.$$
 Using these last two inequalities in the bound for $H^{(1)}_1,$ and substituting into the expression for $\kfunc,$ we obtain
$$|\kfunc(\bx,\by)| \le C' e^{-k|\Im x_1| + k|R|} \left(\frac{1}{\sqrt{k|\bx-(L,0)|}}+\frac{1}{k|\bx-(L,0)|} \right) \frac{\sqrt{2}\sqrt{|x_1+L|^2+H^2}}{|\bx-(L,0)|},$$
for some $C'$ independent of $\bx$ and $\by.$ We note that $|\bx-(L,0)|\ge\delta$ and, for $\Re x_1 \ge 2L,$ $|\bx-(L,0)| \ge |\bx|/2.$ Averaging these bounds yields $|\bx-(L,0)| \ge {(2\delta + |\bx|)}/{4}.$ Similarly, we observe that
$$\sqrt{|x_1+L|^2+H^2} \le \sqrt{2|\bx|^2+2L^2+H^2}.$$

Using these inequalities in the previous bound for $\kfunc,$ we obtain
$$|\kfunc(\bx,\by)| \le C_{1} \frac{e^{-k|\Im x_1|}}{\sqrt{1+|\bx|}},  $$
for some constant $C_{1}$ depending on $\rho, \delta, H, L,$ and $k.$ Using the local bound, possibly increasing $C$ we have that the above inequality holds for all $\bx \in \Gamma_R$ and $\by \in \Gamma_M.$ We omit the proof for $\by \in \Gamma_M$ and $\bx \in \Gamma_L$ which follows almost identical reasoning.

We now turn to the last two cases: $\bx \in \Gamma_M$ and $\by \in \Gamma_R$ or $\Gamma_L.$ We prove only the first, since the latter follows by analogy. The local estimate (for $|\by-(L,0)|<\rho$) follows identical reasoning as for the previous case. For the far part, again employing similar arguments as for the previous case, we arrive at
$$|\kfunc(\bx,\by)| \le e^{-k|\Im y_1| + k|R|} \left(\frac{1}{\sqrt{k|\by-(L,0)|}}+\frac{1}{k|\by-(L,0)|} \right) \frac{\sqrt{2}H}{|\by-(L,0)|}.$$
 provided that $|\by-(L,0)|^2 >2 H^2$ and $|x_1|<L-\delta.$ The change in the factor on the right is due to the fact that for $\by \in \Gamma_R,$ $\bn_\mathbb{C} = (0,-1)$ and $\bx-\by = (x_1-y_1,x_2),$ since the second component $y_2$ of $\by$ is zero. Using similar bounds on $|\by-(L,0)|$ as in the previous case we arrive at the bound
$$|\kfunc(\bx,\by)| \le C_{2} \frac{e^{-k|\Im y_1|}}{(1+|\by|)^{\frac 32}}$$
where $C_2$ again depends on $\rho,\delta,H,L$ and $k.$ Using the local bound, possibly increasing $C_2$ we have that the above inequality holds for all $\bx \in \Gamma_M$ and $\by \in \Gamma_R.$
\end{proof}

\begin{corollary}\label{cor:decay}
    Let $\sigma \in \cspace^{\alpha,\beta},$ $\alpha \in (0,1/2),$ $0<\beta \le k,$ and $\tGamma \in \mathcal{G}_{\bx}.$ Then 
    $$|K_{\tGamma}[\sigma](\bx)| \le C \frac{e^{-\beta |\Im \bx|}}{(1+|\bx|)^{1/2}}\|\sigma\|_{\alpha,\beta}$$
    for all $\bx \in \tGamma.$
\end{corollary}
\begin{proof}
    The bound follows immediately from the previous Lemma. All that is required of $\sigma$ is that it be bounded on $\Gammac.$
\end{proof}

\begin{remark}
    Actually, we have shown something stronger. For any $\sigma \in \cspace^{\alpha,\beta},$ $\alpha < (0,1/2),$ $0<\beta\leq k,$ we have shown that $K[\sigma]$ is well-defined and is in $\cspace^{1/2,k}.$
\end{remark}

We are now ready to prove  \cref{prop:Kdef}.
\begin{proof}[Proof of \cref{prop:Kdef}]

Let $\sigma \in \mathcal{C}^{\alpha,\beta},$ $\alpha \in (0,1/2),$ $0<\beta \le k.$ For any $\bx \in \Gamma_L \cup \Gamma_R$ and any $\tGamma \in \mathcal{G}_{\bx },$ the value of $K_{\tGamma}[\sigma]$ at $\bx$ is determined by an integral over $\tGamma \cap \Gamma_M.$ Since every $\tGamma \in \mathcal{G}_{\bx}$ (or in $\mathcal{G}$ for that matter) is identical on $\Gamma_M$ (agreeing with the real curve $\Gamma$) the independence of $K_{\tGamma}[\sigma]$ on $\tGamma$ is clear. Establishing analyticity is also straightforward, since for $\by \in \Gamma_M$ and $\bx \in \Gamma_R$ or $\bx \in \Gamma_L,$ the argument of the square-root in the function $r(\bx,\by)$ lies in the upper half plane or $[0,\infty)$ and therefore does not hit the branch cut of the square-root. Thus $\kfunc$ is analytic in $x_1,$ the first (and only non-zero) component of $\bx.$ Finally, the required decay follows from \cref{cor:decay}.

The only  part that remains to show is that the values along $\Gamma_M$ are independent of $\tGamma.$ This follows immediately by Cauchy's integral theorem, invoking the analyticity of $\sigma$ and of the kernel $\kfunc(\bx,\by)$ for $\by \in \Gamma_{L} \cup \Gamma_{R}$ and their decay properties.
\end{proof}

\begin{lemma}
    Let $\sigma \in \mathcal{C}^{\alpha,\beta},$ $\alpha \in (0,1/2),$ $0 <\beta \le k.$ Then $(1+|\bx|)^\alpha e^{k|\Im \bx|}K[\sigma]$ is equicontinuous on $\Gammac.$
\end{lemma}
\begin{proof}
    Equicontinuity follows immediately from the continuity of the kernel $\kfunc$ and the bounds in \cref{lem:sqrt}.
\end{proof}

Having established the well-defined nature of the operator $K$ given in \cref{def:K_def}, and its bounded mapping from $\cspace^{\alpha,\beta} \to \cspace^{1/2,k} \subset \cspace^{\alpha,\beta},$ we now turn to the proof of \cref{lem:compact}.

\begin{proof}[Proof of \cref{lem:compact}]
Suppose $\sigma_{n}$, $n=1,2,\ldots$, is a bounded sequence in $\cspace^{\alpha, \beta}$. To prove compactness of $K$, it suffices to show that $f_{n} = K[\sigma_{n}]$ has a Cauchy subsequence in $\cspace^{\alpha,\beta}$. From the results, above, we know that that the $f_{n}$'s are equicontinuous on $\Gamma_{\mathbb{C}}$ and in $\cspace^{1/2, k}$. Let $T_R :=\{\bx \in \Gamma_L\cup \Gamma_R,\,|\bx \pm (L,0)|\le R\}\cup \Gamma_M.$ Then for any $\epsilon$, there exists $R_{\epsilon}$ such that,
$$\sup_{\bx \in \Gamma_L\cup \Gamma_R,\,|\bx \pm (L,0)|>R_\epsilon}|f_{n}(\bx)| e^{\beta|\Im \bx|} (1+|\bx|)^\alpha < \epsilon$$
for all $n$; and $f_{n}$ has a Cauchy subsequence when restricted to $T_{R_{\epsilon}}$ in the $\cspace^{\alpha,\beta}$ topology. The proof then follows using a standard diagonalization argument.
\end{proof}

\subsection{Proof of \cref{thm:inject}}

We begin with the following lemma, which establishes that if $(I+K)[\sigma]$ vanishes on any $\tGamma\in \mathcal{G}$ then $(I+K)[\sigma]$ vanishes on $\Gamma.$

\begin{lemma}\label{lem:op_vanish}
    Suppose  $\sigma \in C^{\alpha,\beta}$ for some $0<\beta \leq k$ and $\alpha \in (0,1/2)$ and let $K[\sigma]$ be the function given in \cref{def:K_def}. If $\sigma + K[\sigma]$ vanishes along any curve $\tGamma \in \mathcal{G}$ then $\sigma + K[\sigma]$ vanishes on $\Gamma.$ 
\end{lemma}
\begin{proof}
Note that by assumption $(I+K)[\sigma]$ is analytic on $\Gamma_{L} \cup \Gamma_{R}.$ If  $\tGamma$ contains a point in the interior of $\Gamma_R$ then analyticity implies that it vanishes everywhere in the interior of $\Gamma_R.$ Since $(I+K)[\sigma]$ is continuous up to its boundary $\Gamma \in \Gamma_R,$ it follows that it vanishes on $\Gamma_R \cap \Gamma.$ Similar reasoning applies to $\Gamma_L.$  
\end{proof}

We also have the following result, which shows that if $(I+K)[\sigma]$ vanishes on $\Gamma,$ then the function $u = \cD_\Gamma [\sigma]$ and its derivatives vanish everywhere in $\Omega.$

\begin{lemma}
    Let $\sigma \in C^{\alpha,\beta}$ for some $\beta >0$ and $\alpha \in (0,1/2)$ and suppose that $(I+K)[\sigma]$ vanishes on $\Gamma.$ Let $u$ be the double-layer potential generated by $\sigma$ in the upper half-space $\Omega,$ i.e. $u(\bx) = \cD_\Gamma[\sigma](\bx)$ for $\bx \in \Omega.$ Then $u \equiv 0$ in $\Omega$. 
\end{lemma}
\begin{proof}
    From \cref{lem:sqrt} and the equation $\sigma=-K[\sigma]$ it follows that, for $x_1\in\Gamma_L\cup\Gamma_R$
    \begin{equation}
        \sigma(x_1)=-\int_{\Gamma_M}g(x_1,\by)\sigma(y)ds(y).
        \end{equation}
        Since $g(x_1,\by)$ has an asymptotic expansion as $x_1\to\pm\infty,$ which holds uniformly for $\by$ in a compact set,
     we conclude that $\sigma$ must also have asymptotic expansions like those in \cref{eq:asymp_f}. Hence \cref{thm:d_pde} implies that the hypotheses of \cref{thm:pde-uniqueness} are also satisfied and therefore $u\equiv 0.$
\end{proof}
Using this last result, and the continuity of the derivative of the double-layer potential across $\Gamma,$ we can also show that $u = \cD_\Gamma[\sigma]$ vanishes in the lower half-space $\mathbb{R}^2\setminus \Omega.$

\begin{lemma}
    Let $\sigma \in C^{\alpha,\beta}$ for some $\beta >0$ and $\alpha \in (0,1/2)$ and suppose that $(I+K)[\sigma]$ vanishes on $\Gamma.$ Let $u$ be the double-layer potential generated by $\sigma$ in the lower half-space $\mathbb{R}^2\setminus\Omega,$ i.e. $u(\bx) = \cD_\Gamma[\sigma](\bx)$ for $\bx \in \mathbb{R}^2 \setminus \Omega.$ Then $u \equiv 0$ in $\mathbb{R}^2\setminus\Omega.$
\end{lemma}

The previous two lemmas imply the following corollary.
\begin{corollary}\label{cor:sig_vanish}
        Let $\sigma \in C^{\alpha,\beta}$ for some $\beta >0$ and $\alpha \in (0,1/2)$ and suppose that $(I+K)[\sigma]$ vanishes on $\Gamma.$ Then $\sigma \equiv 0$ on $\Gamma.$
\end{corollary}

We are now ready to prove \cref{thm:inject}.

\begin{proof}[Proof of \cref{thm:inject}]
Suppose that $\sigma \in \cspace^{\alpha,\beta},$ $0<\alpha <1/2$, $0<\beta \leq k$ and $\sigma + K[\sigma]$ vanishes on any $\tGamma \in \mathcal{G}.$ From \cref{lem:op_vanish}, $\sigma + K[\sigma]$ vanishes on the real contour $\Gamma.$ Then, by \cref{cor:sig_vanish}, $\sigma \equiv 0$ on  $\Gamma.$ In particular, $\sigma \equiv 0$ in $\Gamma_M.$ By construction of the operator $K_{\tGamma},$ see \cref{rem:vanish}, it follows immediately that $\sigma$ vanishes on $\Gamma_L \cup \Gamma_R$ and hence on all of $\Gammac.$ 

\end{proof}

\begin{proof}[Proof of \cref{thm:invert}]
We begin by recalling from \cref{lem:compact} that  for all $\alpha \in (0,1/2),$ and $0<\beta \leq k,$ the operator $K:\mathcal{C}^{\alpha,\beta} \to \mathcal{C}^{\alpha,\beta}$ is compact. In Theorem \ref{thm:inject} it was also shown that $(I+K):\mathcal{C}^{\alpha,\beta} \to \mathcal{C}^{\alpha,\beta}$ is injective. It follows immediately from standard results in Riesz theory, see \cite{kress2014} for example, that $K$ is surjective with a bounded inverse.

\end{proof}

\begin{proof}[Proof of \cref{thm:main}]
The first part of the result has already been shown in \cref{thm:invert}. Suppose now $f \in \cspace^{1/2,k} \subset \cspace^{\alpha,k}$ with $\alpha < 1/2$. Thus, by \cref{thm:invert} there exists $\sigma \in \cspace^{\alpha,k}$ which then satisfies
\begin{equation}
\sigma(\bx) = f(\bx) - K_{\Gamma}[\sigma](\bx) \, .
\end{equation}
Since $f$ and $\kfunc$ both have  asymptotic expansions like those in \cref{eq:asymp_f}, it follows that $\sigma$ has the same  type of asymptotic expansion. Finally, since $f\in \cspace^{1/2,k}$ and $K: \cspace^{\alpha,\beta} \to \cspace^{1/2,k}$ for all $\alpha\in(0,1/2), 0<\beta \leq k$, we conclude that $\sigma$ is also in  $\cspace^{1/2,k}$. 

The fact that $\mathcal{D}_{\Gamma}[\sigma]$ satisfies \cref{eq:bvp} follows from the proof of \cref{thm:d_pde}. Finally, using estimates similar to the ones derived in \cref{lem:sqrt}, it is easy to show that for $\sigma \in \cspace^{1/2,k}$, and $\bx\in \Omega \cap \{ |x_{1}| \leq L \}$, that
$\mathcal{D}_{\Gamma}[\sigma](\bx) = \mathcal{D}_{\tGamma}[\sigma](\bx)$ for all $\tGamma \in \mathcal{G}$.

\end{proof}


\section{Extension to three dimensions}
The results presented in~\Cref{sec:main_res} extend naturally to solutions of the Helmholtz Dirichlet scattering problem on
perturbed half-spaces in three dimensions as well. The key ingredient for the method is the outgoing nature
of solutions of the associated integral equation.
As in the two dimensional case, the solution of the Helmholtz Dirichlet problem in three dimensions can be expressed
using a double layer potential with density $\sigma$, denoted by $\cD^{3d}_{\Gamma}[\sigma]$, and given by
\begin{equation}
u(\bx) = \cD^{3d}_{\Gamma}[\sigma](\bx) = \int_{\Gamma} \frac{(\bx-\by) \cdot \bn(\by)}{4 \pi. |\bx - \by|^3}(1 - ik |\bx - \by|) \exp(ik |\bx - \by|) \sigma(\by) ds_{\by} \, .
\end{equation}

Upon imposing the boundary condition, and using the standard jump relations for the Helmholtz double layer potential~\cite{colton}, we obtain the following integral equation for the density $\sigma$, 
\begin{equation}\label{eqn:first_BIE_3D}
-\frac{1}{2} \sigma(\bx) + D^{3d}_{\Gamma}[\sigma](\bx) = f(\bx) \, , \quad  \bx \in \Gamma \,  ,
\end{equation}
where $D^{3d}_{\Gamma}[\sigma]$ is the principal value part of the double layer potential in three dimensions restricted to the boundary $\Gamma$. We assume that the surface $\Gamma$ is flat outside of a ball of radius $L$ centered at the origin, and that in the flat region, the surface is parametrized as $\{(x_{1}, x_{2}, 0):\:\sqrt{x_{1}^2+x_{2}^2} >L\}$. As before, if both $\bx$ and $\by$ lie outside the ball of radius $L,$ then the kernel function for $D^{3d}_{\Gamma}$ vanishes.

The outgoing nature of the Green's function implies that the solution $\sigma$ to the integral equation above, and the data $f$ satisfy the estimate 
\begin{equation}
\sigma(\bx), f(\bx) \lesssim \frac{e^{ik |\bx|}}{|\bx|} \, ,
\end{equation}
as $|\bx| \to \infty$  for outgoing data generated either by a collection of point/volumetric sources or incident plane waves/wave packets. If we replace  $|\bx|$ by $$r(\bx) = \sqrt{x_{1}^2 + x_{2}^2 + x_{3}^2},$$
then we note that the integral equation
can be deformed to a complex surface $\tilde{\Gamma}$ with the coordinates complexified independently as in the 2d-case, i.e.  the complexified surface is parameterized by 
$$\{(x_{1} + i \eta(x_1), x_{2} + i \eta(x_2), 0):\: x_1^2+x_2^2>L^2\}$$
where $\eta(t)$ is a monotonic function that is non-zero only outside the support of the perturbation, i.e. $\eta(t) = 0$ for $|t| \leq L$,  and $\pm\eta(t) >0$ for $\pm t>L$. 

The proof of this result relies on the following key components: 
\begin{enumerate}
    \item Uniqueness of solutions to the differential equation with Dirichlet and Neumann boundary conditions under the assumption of solutions being uniformly outgoing (\Cref{thm:pde-uniqueness}).
    \item The analogue of \Cref{thm:d_pde} showing that, in 3D, layer potentials with `outgoing densities' result in outgoing solutions. \item The analogue of \Cref{lem:rt_bnd} giving estimates for the decay rate of the Green's function, and outgoing data on the complex surface $\tilde{\Gamma}.$ 
    \item The analogue of \Cref{lem:compact}, proving the compactness of the integral operator from $\cspace^{\alpha, \beta}$ to itself for $\alpha \in (0,1/2)$, and $0 < \beta \leq k.$ 
    \item Finally analogues of \Cref{thm:inject,thm:invert}, establishing the analyticity properties of the Green's function, the boundary data, and the integral operator 
\end{enumerate}

The uniqueness of solutions to Dirichlet and Neumann boundary value problems is proved in~\cite{odeh1963uniqueness}. The proof that the outgoing densities result in outgoing solutions in three dimensions turns out to be rather technical and will be published in a forthcoming paper. The rest of the results extend to three dimensions modulo the following minor caveats.

First, due to the decay of the Green's function, the results extend to $\alpha \in (0,1)$ as opposed to $\alpha \in (0,1/2)$. Secondly, the proof of compactness, which relies on the estimates of the kernel of the integral operator, needs to be suitably modified. 
In particular, the kernel is not continuously differentiable in three dimensions for $\bx, \by \in \Gamma_{M}$. However, the singularity is weak and does not affect the proof of compactness from $\cspace^{\alpha, \beta}$ to itself. Moreover, it is easier to prove the estimate that $g(\bx, \by)$, for fixed $\by$, and $\bx \in \Gamma_{L} \cup \Gamma_{R}$, satisfies an estimate with an exponential decay of $\exp{(-(k-\delta) |\Im{x_{1}}|)}$ for any $\delta>0$, as opposed to $\exp{(-k |\Im{x_{1}}|)}$ (with a similar result holding for fixed $\bx$, and $\by \in \Gamma_{L} \cup \Gamma_{R}$). A consequence of this weaker estimate is that the integral operator can be shown to be compact from $\cspace^{\alpha,\beta}$ to itself for $\alpha \in (0,1)$, and $0<\beta<k$. The compactness of the integral operator in this space, has no affect on the remainder of the proof.

\section{Error analysis}
In this section, we detail our numerical approach to discretize and solve the integral equation~\cref{eq:on_curve} for a suitable choice of $\Gamma_{\mathbb{C}}\in \mathcal{G}$ and bound the error in the solution when computed using the truncated integral equation.  

\subsection{Choice of the contour, cutoff, and discretization}
We note from \cref{thm:invert} that for any $\tGamma \in \cG,$ \cref{eq:on_curve} has a unique solution. The advantage of choosing a $\tGamma \in \cG \neq \Gamma$ lies in truncation,  since for $f \in \mathcal{C}^{1/2,k},$ the density is also in $\sigma \in \mathcal{C}^{1/2,k}$. Let $\bx = (x_{1},x_2) \in \tGamma$. For any given $\eps$, we truncate the curve $\tGamma \in \cG$ at $x_{1} = \pm \Ltrunc$, so that $k|\Im{x_{1}}| > \log{(1/\eps)}$ for $|\Re{x_{1}}| > \Ltrunc$. This truncation implies that $|\sigma(\bx)|, |f(\bx)| < C \eps$, for some constant $C$, and for all $|\Re(x_{1})| > \Ltrunc$. 
Let $\tGammat \subset \tGamma$ denote the truncated curve $\tGamma \cap \{ |\Re{x_{1}}| \leq \Ltrunc \}$. The truncation of surfaces in the three dimensional case follows, \emph{mutatis mutandis}. One should think of $\eps$ as ``machine-epsilon."

In two dimensions, the curves are discretized using the boundary integral equation package {\tt chunkie}~\cite{chunkie}, which uses piecewise Legendre polynomial expansions to represent boundary curves and densities. For all experiments in this section, we use $16$th order expansions for representing both the boundary and the densities. The kernel of the double layer potential is weakly singular and is computed using locally corrected quadrature rules, which use standard Gauss-Legendre quadrature for well-separated points where the kernel is smooth, and specialized generalized Gaussian quadrature~\cite{bremer_2012b} rules for nearby points where the kernel is weakly-singular. The generalized Gaussian quadrature rule extends as is to curves $\tGamma \in \cG$ with non-zero imaginary parts, as complexifying the coordinates does not affect the singularity of the kernel of the integral operator.

In three dimensions, even though the coordinate complexification method extends in a trivial manner, the design of appropriate quadrature for the associated weakly singular integrals is a more subtle question. In particular, the standard locally corrected quadrature rules which rely on generalized Gaussian quadratures or high-order corrected trapezoid rules require suitable modification to allow for first fundamental forms which are complex symmetric as opposed to being positive definite. We follow the discretization procedure and high-order corrected trapezoid quadratures discussed in~\cite{hoskins2023quadrature}, which relies on an equispaced discretization of the geometry and adapts the high-order corrected trapezoid quadratures of~\cite{wu2020corrected} for complex symmetric first fundamental forms. 

For our numerical examples, we define suitable smooth, complexified contours $\tGamma$ using the function 
\begin{equation}
\phi_{i, \alpha}(x) = \frac{\alpha}{3} \left( \phi(3(x + L + t_{0})) -  \phi( 3(L+t_{0} - x))\right)
\end{equation} 
with $t_{0} = \sqrt{\log{(1/\varepsilon)}}/2$, and $\phi(x)$ given by
\begin{equation}
\phi(x) = -\frac{1}{2}\int_{x}^{\infty} \erfc(t) \, dt = \frac{1}{2} \left( x \erfc{(x)} - \frac{\exp{(-x^2)}}{\sqrt{\pi}} \right) \, .
\end{equation}
These choices guarantee that $|\phi_{i, \alpha}(x)| < \varepsilon$ for $|x| < L$, and that $\phi_{i, \alpha}$ to precision $\varepsilon$ is
$\alpha( x + L + t_{0})$ for $x < -L -2 t_{0}$, and $\alpha (x - L - t_{0})$ for $x > L + 2t_{0}$. In several of our examples, we will also employ a suitable {\it bump function} to guarantee that the $x_2$-component of $\tGamma$ is sufficiently small when $\Im x_1$ is non-zero. Here we use the function $\psi_{L - \delta}(x)$ given by
\begin{equation}
\psi_{s}(x) = 2 - \left(\erfc{(2(x + s - t_{0}))} + \erfc{(2(s - x - t_{0})}) \right) \, .
\end{equation}
For these parameters, $|\psi_{L - \delta}(x)| < \eps$ for all $|x| > L-\delta$, and $|\psi_{L-\delta}(x) - 2| < \eps$ for all $|x| < L - \delta - t_{0}$, with $t_{0} = \sqrt{\log{(1/\varepsilon)}}/2$.

\subsection{Truncation error analysis} 
\label{subsec:trunc-err}
Suppose that the complexified curve $\tGamma$ with contour $\phi_{i,1}$ is truncated at $\Ltrunc = L + t_{0} + \max{(t_{0}, \log{(1/\sqrt{\eps})}/k)}$ and the truncated curve is denoted by $\tGammat$. Furthermore, suppose that $\tsigma$ is a solution to the truncated integral equation
\begin{equation}
\label{eq:trunc-eq} 
-\frac{1}{2} \tsigma(\bx) + \cD_{\tGammat}[\tsigma](\bx) = f(\bx) \, , \quad \bx \in \tGammat \, .
\end{equation}
We would like to show that there exists an $0<\eps_{0}$ such that, for all $\eps < \eps_{0}$, the truncated integral equation~\cref{eq:trunc-eq} is invertible and $\|\tsigma-\sigma\|_{\alpha,k;\tGamma}=O(\eps),$   for any $\alpha < 1/2$ (we assume that $\alpha < 1/2$ for the rest of this analysis). For $\tGamma\subset\Gamma_{\bbC},$ here and in the sequel we let
\begin{equation}
    \|f\|_{\alpha,\beta;\tGamma}:=\sup\limits_{\bx\in\tGamma}e^{\beta|\bx|}(1+|\bx|)^{\alpha}|f(\bx)|.
\end{equation}

\begin{lemma}
There exists $0<\eps_{0}$, such that for all $\eps<\eps_{0}$, and $\Ltrunc = L + t_{0} + \max{(t_{0} + \log{(1/\sqrt{\eps})/k})}$, such that the integral operator $-1/2 + D_{\tGammat}$ is invertible on the restriction of $\cspace^{\alpha, k}$ to $\tGammat$.
\end{lemma}
\begin{proof}
Let $\Gamma_{1} = \Gamma_{M}$, $\Gamma_{2} = \tGammat \setminus \Gamma_{M}$, and $\Gamma_{3} = \tGamma \setminus \tGammat$, and let $K_{j,\ell}$ denote the corresponding blocks of the integral operators mapping functions in $\cspace^{\alpha, k}$ restricted to $\tGamma_{\ell}$ to functions in $\cspace^{\alpha, k}$ restricted to $\tGamma_{j}$. Let $\sigma_{j}$ denote the restrictions of $\sigma \in \cspace^{\alpha, k}$ to $\Gamma_{j}$, $j=1,2,3$. Owing to the properties of the kernel $g_{k}$ (see~\cref{lem:sqrt}), the integral equation on $\tGamma$ takes the form
\begin{equation}
\left(-\frac{1}{2} I 
+
\begin{bmatrix}
K_{1,1} & K_{1,2} & K_{1,3} \\
K_{2,1} & 0 & 0 \\
K_{3,1} & 0 & 0  
\end{bmatrix}
\right)
\begin{bmatrix}
\sigma_{1}(\bx) \\
\sigma_{2}(\bx) \\
\sigma_{3}(\bx)
\end{bmatrix} = 
\begin{bmatrix}
f_{1}(\bx) \\
f_{2}(\bx) \\
f_{3}(\bx)
\end{bmatrix}
\end{equation}
\Cref{thm:invert} states that the above equation is invertible. Since the $(3,3)$ block of the integral equation is just the identity operator, using the Schur complement, we conclude that the operator
\begin{equation} 
\left(-\frac{1}{2} I 
+
\begin{bmatrix}
K_{1,1} + 2 K_{1,3} K_{3,1} & K_{1,2} \\
K_{2,1} & 0   
\end{bmatrix}
\right)
\end{equation}
is invertible with a bounded inverse in the $\cspace^{\alpha, k}$ topology. This operator differs from the truncated integral equation by the $2K_{1,3} K_{3,1}$ term in the $(1,1)$ block. To prove invertibility of the truncated integral equation, it suffices to show that $2K_{1,3} K_{3,1}$ in norm is $O(\varepsilon)$. Note that the $\cspace^{\alpha,k}$ topology restricted to $\Gamma_{1}$ is equivalent to the topology of continuous functions with the $\sup$ norm. Let $\sigma$ be a continuous function on $\Gamma_{1}$ with sup norm $1$. Then
$\mu = K_{3,1}[\sigma]$ is given by
\begin{equation}
\mu(\bx) = \int_{-L}^{L} g_{k}(\bx, \by) \sigma(\by) \, dy \, \quad \bx \in \tGamma \setminus \tGammat \,.
\end{equation}
Using the bound on $g_{k}(\bx, \by)$ in~\cref{lem:sqrt}, it follows that there exists $C_{1}$ such that
\begin{equation}
\label{eq:mubound}
|\mu(\bx)| \leq 2 L C_{1} e^{-k |\Im(x_{1})|} \, , \quad \bx \in \tGamma \setminus \tGammat \, .
\end{equation}
Let $\tau = K_{1,3}[\mu] = K_{1,3} K_{3,1} [\sigma]$, then $\tau$ is given by
\begin{multline}
\tau = \\ \int_{-\infty}^{-\Ltrunc} + \int_{\Ltrunc}^{\infty} g_{k}(\bx, (y_{1} + i\phi_{i,1}(y_{1}), 0)) \mu((y_{1} + i\phi_{i,1}(y_{1}), 0)) \sqrt{1 + \phi_{i,1}'^2(y_1)}  dy_{1} \,,
\end{multline}
for $\bx \in \Gamma_{M}$.
For $\by \in \tGamma \setminus \tGammat$, \cref{lem:sqrt} implies that $|g_{k}(\bx, \by)| \leq C_2 e^{-k |\Im y_1|}$.
Finally, the choice of the contour implies $\phi_{i,1}(y_{1}) \geq y_{1} - L - t_{0}- \varepsilon$ for $y_{1} > L + 2t_{0}$, and that $|\sqrt{1+ \phi_{i,1}'^2}| \leq C_{3}$ for some constant $C_{3}$. Combining these estimates, and the bound on $\mu$ in~\cref{eq:mubound}, we get that there exists $C$ such that
\begin{equation}
\begin{aligned}
\int_{\Ltrunc}^{\infty}  &\left|g_{k}(\bx, (y_{1} + i\phi_{i,1}(y_{1}), 0)) \mu((y_{1} + i\phi_{i,1}(y_{1}), 0)) \sqrt{1 + \phi_{i,1}'^2(y_1)} \right| dy_{1} \leq \\
&\tilde{C} \int_{\Ltrunc}^{\infty} e^{-2k \phi_{i,1}(y_{1})} dy_{1} \\
&\leq \tilde{C}\int_{\Ltrunc}^{\infty} e^{-2k (y_{1} - L - t_{0} - \varepsilon)} dy_{1} \leq C e^{-2k \log{(1/\sqrt{\varepsilon})}/k} \leq C \varepsilon
\end{aligned}
\end{equation}
The bound for the contribution to $\tau$ from $(-\infty, -\Ltrunc)$ follows in a similar manner. Thus we conclude that $K_{1,3} K_{3,1}$ in the $\sup$-norm topology is $O(\varepsilon)$, and hence there exists an $0<\varepsilon_0$  so that the truncated integral equation is invertible for $\varepsilon<\varepsilon_0$.
\end{proof}

Note that even though we solve for $\tsigma$ on $\tGammat$, it can be extended to $\tGamma$ via the formula
\begin{equation}
\label{eq:tsig_ext}
\tsigma(\bx) = -2f(\bx) + 2\cD_{\Gamma_{M}}[\tsigma](\bx)\, ,
\end{equation} 
since the kernel $g_{k}(\bx, \by)$ is $0$ for $\bx, \by \in \tGamma \setminus \Gamma_{M}$.

\begin{lemma}
Suppose $\alpha<1/2$, and $\tsigma$ satisfies~\cref{eq:trunc-eq}, and for $\tGammat$ truncated at $\Ltrunc$ with $\eps < \eps_{0}$. Then $|\tsigma - \sigma|_{\alpha, k} < C\eps |f|_{\alpha,k}$ 
\end{lemma}
\begin{proof}
 By linearity, it suffices to take $f$ with norm 1.
Let $v = \sigma - \tsigma$, using the properties of the kernel $g_{k}$, $v$ satisfies
\begin{multline}\label{int-eqn:untrunc}
-\frac{1}{2} v(\bx) + \cD_{\tGamma}[v](\bx) =\\ \begin{cases}
\cD_{\Gamma_{3}} [\tsigma](\bx) = -2 K_{1,3}[f] + 2 K_{1,3} K_{3,1}[\tsigma](\bx) \, ,\quad &\bx \in \Gamma_{M} \,,\\
0 \, ,\quad &\bx \in \tGamma \setminus \Gamma_{M} \,,
\end{cases} \
\end{multline}
where $\tsigma$ on $\Gamma_{3}$ is defined via~\cref{eq:tsig_ext}.  Using an estimate similar to the one for $K_{1,3} K_{3,1}[\tsigma]$, it can be shown that $K_{1,3}f$ is also $O(\varepsilon)$ for $f\in \cspace^{1/2,k}$. 
Then, the data for the integral equation is clearly in $\cspace^{1/2,k}$ (continuity follows from the flatness of the geometry for $|x_{1}| \in (L-\delta, L]$), with $\cspace^{\alpha, k}$ norm $O(\varepsilon)$. Since the integral equations,~\eqref{int-eqn:untrunc} and~\eqref{eq:trunc-eq}  have uniformly bounded inverses in the $\cspace^{\alpha, k}$ topology, for $\varepsilon<\varepsilon_0,$ we conclude that $v$ is $O(\eps)\|f\|_{\alpha,k}$ in the $\cspace^{\alpha, k}$. 
\end{proof}

Finally, we would like to show that for $x \in \Omega_{E} = \Omega \cap \{|x_{1}|<L-\delta_{1}\}$, for any $\delta_{1} \in (0,\delta]$, the solution computed via the truncated density $\tilde{\sigma}$ is an $O(\eps)$ approximation to the solution $u$. It suffices to show that for $v = \sigma - \tsigma$ which is $O(\eps)$ in the $\cspace^{\alpha, k}$ topology, that $\cD[v](\bx)$ for $\bx \in \Omega_{E}$ is $O(\eps)$. 
\begin{lemma}
Suppose that the boundary $\Gamma_{M}$ is of class $\mathcal{C}^{2}$, and that $v$ satisfies~\cref{int-eqn:untrunc}, then there exists $C>0$, such that $|\cD_{\tGamma}[v](\bx)| \leq C \eps$ for all $\bx \in \Omega_{E}$.
\end{lemma}
\begin{proof}
We first show that $v$ is H\"{o}lder continuous on $\Gamma_{M}$. This follows from the facts that, (a) the operators $K_{1,3}$, and $K_{1,3}K_{3,1}$ map restrictions of functions in $\cspace^{\alpha, k}$ to H\"{o}lder continuous functions on $\Gamma_{M}$, and (b) $\cD_{\Gamma_{M}}$ maps continuous functions of $\Gamma_{M}$ to H\"{o}lder continuous functions on $\Gamma_{M}$ as well. Thus,
\begin{equation}
    v(\bx) = -2\cD_{\tGamma}[v] + 4 K_{1,3}[f](\bx) + 2 K_{1,3}K_{3,1}[\tsigma](\bx) \, ,
\end{equation}
for $\bx \in \Gamma_{M}$ is H\"{o}lder continuous. 

Suppose that $\psi$ is a partition of unity on $\tGamma$ which is $1$ on $\{|x_{1}| \leq L - 2\delta_{1}/3\}$ and $0$ for $\{|x_{1}|\geq L-\delta_{1}/3\}$. Let $T=\Omega_{E} \cap |x_{2}|\leq 2H$ be a tubular neighborhood of the boundary, where $H$ as before is the maximum $x_{2}$ extent of $\Gamma_{M}$. For $\cD_{\tGamma}[v \psi](\bx)$, and $\bx \in T$, the estimate follows from boundedness of $\cD$ from 
H\"{o}lder continuous functions on the boundary to H\"{o}lder continuous functions in the bulk. 
The contributions $\cD_{\tGamma}[v \psi](\bx)$ and $\bx \in \Omega_{E} \setminus T$, and $\cD_{\tGamma \cap \Gamma_{M}}[v(1-\psi)](\bx)$, and $\bx \in \Omega_{E}$, the estimate follows from an $L^{\infty}$ norm on the kernel, and $v$ being $O(\eps)$. All that remains is to estimate $w = \cD_{\tGamma \setminus \Gamma_{M}}[v](\bx)$, since $\psi = 0$ for $\bx \in \tGamma \setminus \Gamma_{M}$.

\begin{equation}
w = \int_{-\infty}^{-L} + \int_{L}^{\infty} g_{k}(\bx, (y_{1} + i\phi_{i,1}(y_{1}),0)) v((y_{1} + i\phi_{i,1}(y_{1}),0)) \sqrt{1+\phi_{i,1}'^2(y_{1})} dy_{1}
\end{equation}
Since $v$ is $O(\eps)$ in the $\cspace^{\alpha,k}$ topology, there exists $C>0$
such that $|v(y_{1} + i\phi_{i,1}(y_{1}),0)| \leq C \eps e^{-\Im{\phi_{i,1}(y_{1)}}}$ for $|y_{1}|>L$. Note that $\Im{(r(\bx,\by))} = \Im{(\sqrt{(x_{1} - y_{1} - i\phi(y_{1}))^2 + x_{2}^2})} \geq 0$, and $|r(\bx,\by)|>0$, for $\bx \in \Omega_{E}$, and $\by \in \tGamma \setminus \Gamma_{M}$, which guarantees that $g_{k}(\bx,\by)$ is uniformly bounded there. The imaginary part is non-negative since $x_{1} - y_{1} - i \phi(y_{1})$ is in the first or third quadrant, which implies its square is in the upper half-plane, and so is $x_{2}^2$. This bound combined with the estimate on $v$ for $|y_{1}|>L$ completes the proof. 
\end{proof}
\section{Numerical examples}
In the following, we illustrate the performance of our approach through several numerical examples. We consider $f \in \mathcal{C}^{1/2,k}$ in two dimensions and $f \in \mathcal{C}^{1,k}$ in three dimensions.
\subsection{Accuracy illustration in two dimensions}
Suppose that the interface is parametrized by
\begin{equation}
    \begin{bmatrix} x_{1}(t) \\ x_{2}(t) 
    \end{bmatrix} 
    = \begin{bmatrix}
    t \\
    \exp{(-t^2/4)} \cos{(t)} \psi_{L-\delta}(t)  
    \end{bmatrix} \, , \quad t \in (-\infty, \infty) \, ,
\end{equation}
with $L = 10$, and $\delta = 1$. The truncated complexified contour $\tGammat$ is parametrized as
\begin{equation}
    \begin{bmatrix} x_{1}(t) \\ x_{2}(t) 
    \end{bmatrix} 
    = \begin{bmatrix}
    t  + i\phi_{i, 1}(t) \\
    \exp{(-t^2/4)} \cos{(t)} \psi_{L-\delta}(t)  
    \end{bmatrix} \, , \quad t \in [-\Ltrunc, \Ltrunc] \, ,
\end{equation}
with $\Ltrunc = L + t_{0} + \max{(t_{0}, \log{(1/\sqrt{\eps})}/k)}$. 

To test the accuracy of our approach, we first consider an analytical solution test. Suppose that the boundary data $f(\bx)$ is given by $H_{0}^{(1)}(k r(\bx, \bx_{0}))$, with $\bx_{0} = (1/10, - 1)$, which lies below $\Gamma,$ and $k = 2\pi$. By uniqueness, the solution $u$ in the upper half plane is also given by $H_{0}^{(1)}(k r(\bx, \bx_{0}))$. In the left panel of~\cref{fig:2d-analytical-solution}, we plot the solution in the upper half plane, and the middle panel shows the $\log$ of the pointwise error in the computed solution. The solution in the volume is computed using the smooth quadrature rule for the panels in $\Gamma_{L} \cup \Gamma_{R}$, and locally corrected quadrature based on adaptive integration for panels in $\Gamma_{M}$. The solution has more than $12$ digits of accuracy for $|x_{1}| \leq L$, excluding a small region close to $\bx = (\pm L,0)$, which can be attributed to the quadrature errors due to near singular interactions from panels in $\Gamma_{L}$, and $\Gamma_{R}$.  The right panel in~\cref{fig:2d-analytical-solution} shows the extension of the solution, $\sigma,$ to the integral equation into part of the region $\Gamma_{R}$, which in this case can be evaluated using its values on $\Gamma_{M}$ and the identity 
\begin{equation}
\sigma(\bx) = -2 ( f(\bx) - \cD_{\Gamma_{M}}[\sigma|_{\Gamma_{M}}])(\bx) \, ,\quad \bx \in \Gamma_{R} \,,
\end{equation}
and note that $\sigma(\bx)$ has the expected asymptotic behavior  in $\Gamma_{R}$ as given by~\cref{eq:sig_asymp}.

\begin{figure}[h!]
\centering
\includegraphics[width=\textwidth]{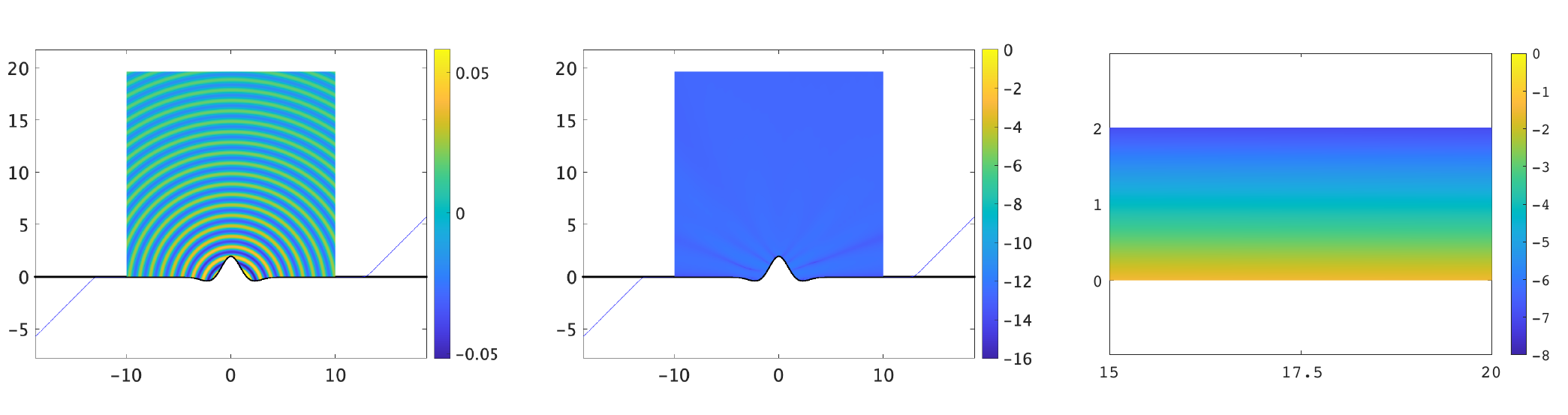}
\caption{Solution for data generated by a point source at $(1/10,-1)$. {\bf Left:} the real part of the scattered field. {\bf Middle:} the log (base 10) error computed using the analytic solution (in this case, the standard Helmholtz Green's function). In both cases the black line denotes the real part of the interface, and the blue line is its imaginary part. {\bf Right:} $\log_{10}|\sigma|$ in part of $\Gamma_R.$  }
\label{fig:2d-analytical-solution}
\end{figure}

To determine the impact of coordinate complexification, for the same boundary data, we measure the error at a single point $\bx = (-1,2)$ in the upper half plane as we vary the truncation point and the slope of the complexified contour. Suppose that the complexified contour $\tGammat$ is parametrized as above but with $\phi_{i, \alpha}$ as the imaginary part instead of $\phi_{i,1}$, and that $\alpha$ and $\Ltrunc$ are varied so that $x_{1}(\Ltrunc) = \eta_{r} + i \eta_{i}$ covers the rectangle $\eta_{r} \in [15,20]$, and $\eta_{i} \in [0,2]$. The result is plotted in~\cref{fig:small_analytic_endpoint}. The accuracy of the solution scales approximately as $O(e^{-\eta_{i} k})$, thus illustrating that the imaginary part of truncated complexified curve must scale like $\log{(1/\eps)}/k$  in order to achieve precision $\varepsilon$.

\begin{figure}[h!]
\centering
\includegraphics[width=0.45\textwidth]{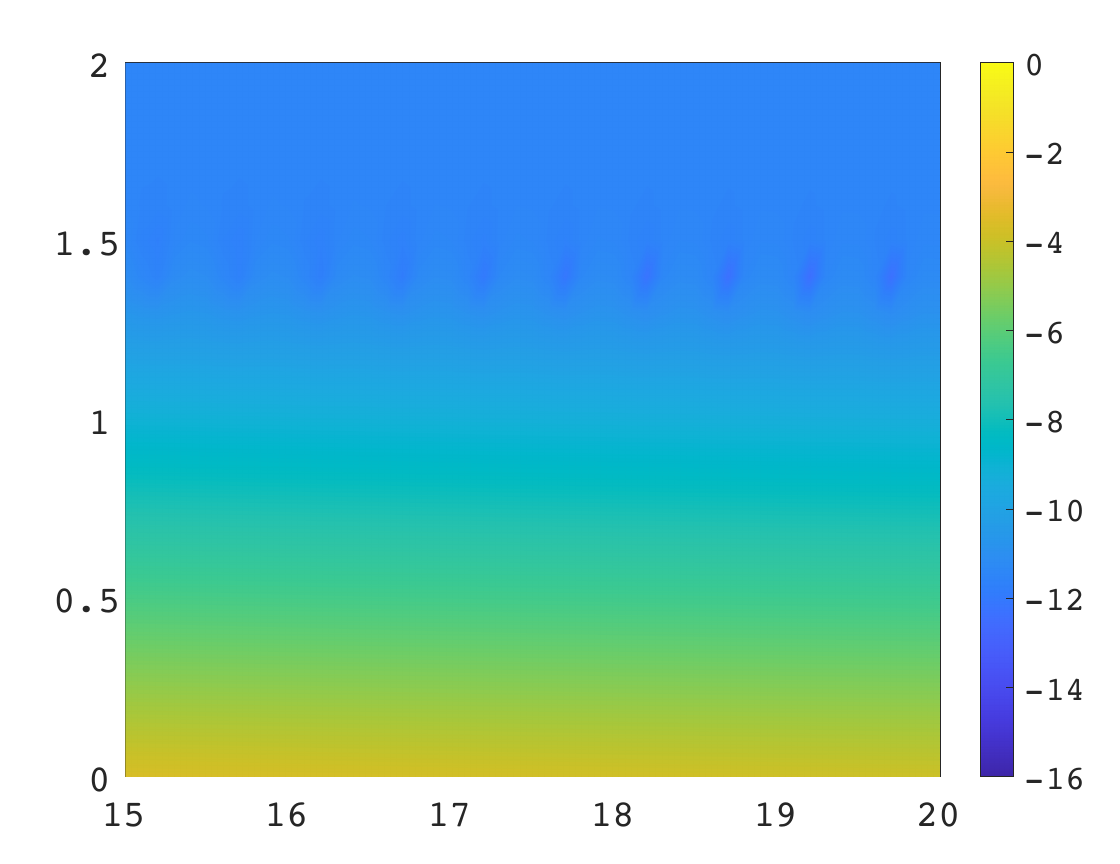}
\caption{Dependence of the error on the endpoints of the solution contour $\tilde{\Gamma}.$ The horizontal and vertical axis correspond to the real and imaginary parts of the endpoint of $\tilde{\Gamma}$ in $\Gamma_R$ ($\Im \,(\tilde{\Gamma})$ was chosen to be odd). The color denotes the $\log_{10}$ relative error.}
\label{fig:small_analytic_endpoint}
\end{figure}

In~\cref{fig:small_pwave_test}, we plot the scattered field from the interface due to an incident incoming plane wave at $\theta = -\pi/4$. Note that the reflected plane wave from the flat contour is also accounted for in the incident plane wave, i.e.
\begin{equation}
\uin = e^{ik \left(x_{1} \cos{(\theta)} + x_{2} \sin(\theta) \right)} - e^{ik \left(x_{1} \cos{(\theta)} - x_{2} \sin(\theta) \right)} \, ,
\end{equation}
and $f = -\uin$. Including the reflected wave ensures that $f(\bx) \equiv 0$ for $|x_{1}|> L$, and thus remains $0$ on the complexified contour as well (as opposed to a plane wave which grows exponentially in $\Gamma_{L} \cup \Gamma_{R}$). We also plot the error in the solution at $\bx = (-1, 2)$ as a function of the number of points used to discretize the interface. In this example, we use $\Ltrunc = L + 8 > L + t_{0} + \max{(t_0, \log{(1/\sqrt{\eps})}/k)}$, so that the length of the panels on $\Gamma_{L}, \Gamma_{M}$, and $\Gamma_{R}$ scale in a commensurate manner. Since an analytical solution is not known for this case, the error is measured by comparing the solution to a reference solution computed with the boundary discretized using $N=2016$ points. The error decays at the expected rate of $O(N^{-16})$.

\begin{figure}[h!]
\centering
\includegraphics[width=\textwidth]{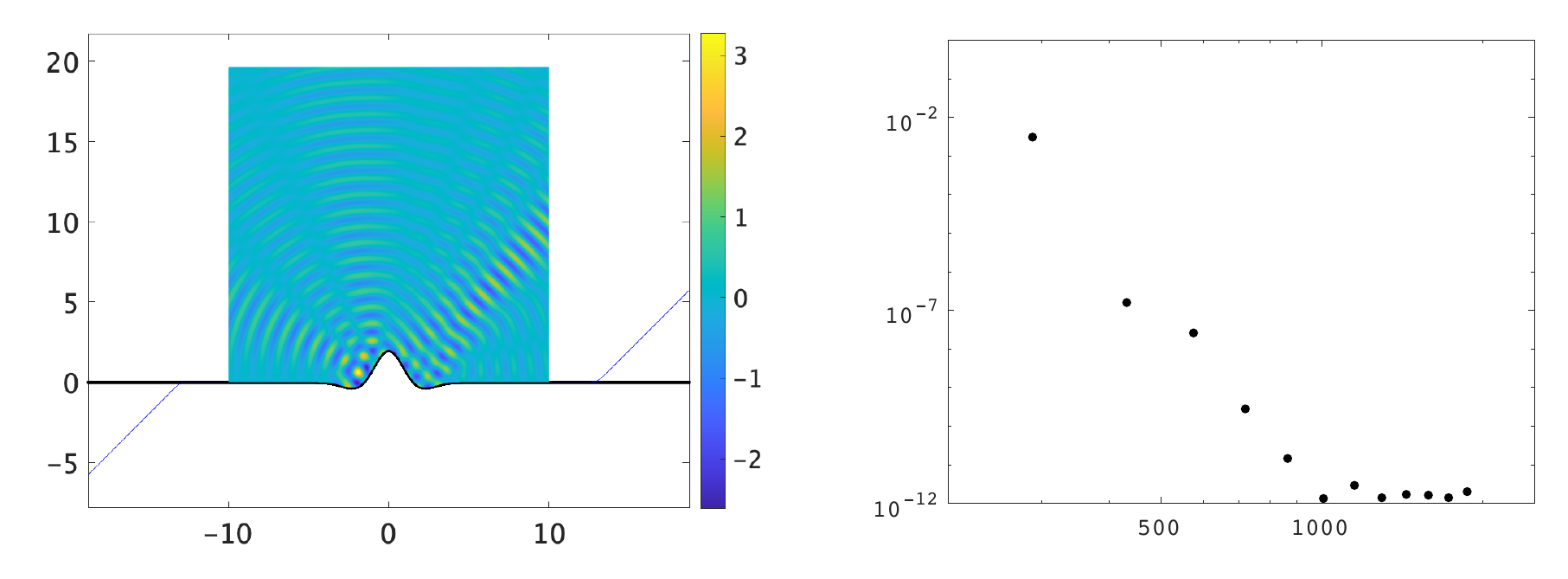}
\caption{Solution of a Helmholtz Dirichlet problem with wavenumber $k=2\pi$ in a half-space. The data was generated from a plane source at $-45^{\rm o}$ to the interface. The reflected plane wave (at $45^{\rm o}$) is also subtracted. {\bf Left:} real part of the scattered field. {\bf Right:} Self-convergence plot of the error at $(-1,2)$ as a function of number of points used in the discretization.}
\label{fig:small_pwave_test}
\end{figure}

\subsection{Acoustic scattering from complicated two dimensional interface}
In this section, we illustrate the performance of the coordinate complexification approach on a more complicated interface due to an incoming point source, and a Gaussian beam. The interface in this setup is parametrized as
\begin{equation}
    \begin{bmatrix} x_{1}(t) \\ x_{2}(t) 
    \end{bmatrix} 
    = \begin{bmatrix}
    t \\
    h(t)  \psi_{L-\delta}(t)  
    \end{bmatrix} \, , \quad t \in (-\infty, \infty) \, ,
\end{equation}
with 
\begin{equation*} 
h(t) = 2\cos{(2\pi t + 0.8)} - \cos{(10\pi/13t + 1.5)} + 0.6 \cos{(\sqrt{2} \pi t + 3.2)} \, ,
\end{equation*} 
$L = 20$, and $\delta = 1$. The truncated complexified contour $\tGammat$ is parametrized as
\begin{equation}
    \begin{bmatrix} x_{1}(t) \\ x_{2}(t) 
    \end{bmatrix} 
    = \begin{bmatrix}
    t  + i\phi_{i, 1}(t) \\
    h(t) \psi_{L-\delta}(t)  
    \end{bmatrix} \, , \quad t \in [-\Ltrunc, \Ltrunc] \, ,
\end{equation}
with $\Ltrunc = L + t_{0} + \max{(t_0, \log{(1/\sqrt{\eps})}/k)}$. 

We compute the solution due to a point source in the upper half plane placed at $\bx_{0} = (0.1, 10)$, i.e.
$f(\bx) = -H_{0}^{(1)}(k r(\bx, \bx_{0}))$, and due to a Gaussian beam with source located at $\bx_{g} = (-10 + 4i; 20 - 8i)$, i.e.
$f(\bx) = -1.5 \times 10^{-23} H_{0}^{(1)}(k r(\bx, \bx_{g})))$. Note that the Gaussian  beam is scaled so that the incident field is $O(1)$ for $\bx \in \mathbb{R}^{2}$. The accuracy of the computed solutions determined through an analytical solution test is at least $10^{-11}$. In the left and middle panels of~\cref{fig:pt_src}, we plot the real part of the incident field, and the real part of the scattered field; the right panel shows the absolute value of the total field for the point source boundary data. In~\cref{fig:gb}, we plot the corresponding results for the Gaussian beam. For both of the examples, we use $k = 2\pi$ as before.

\begin{figure}[h!]
\centering
\includegraphics[width=\textwidth]{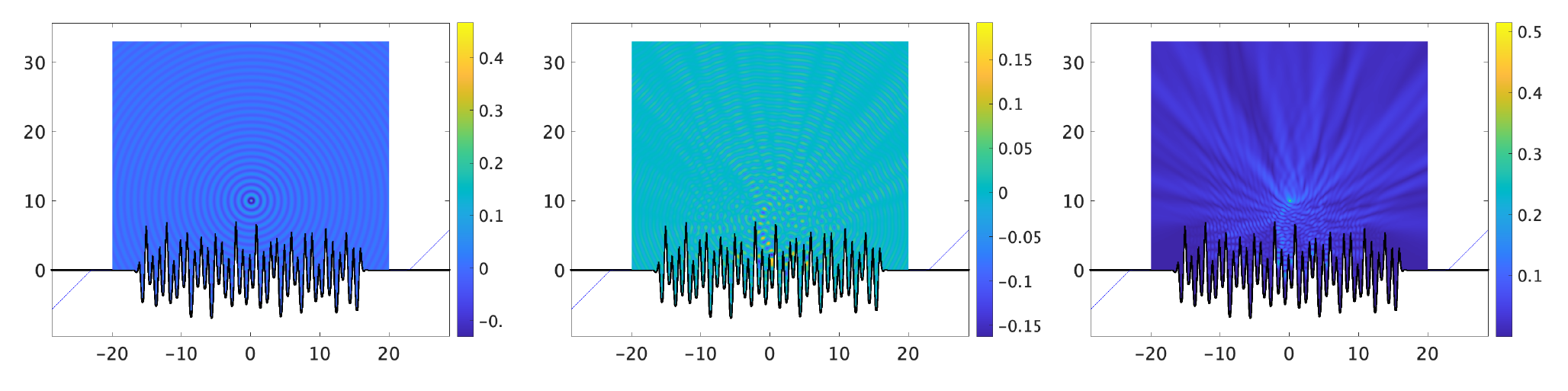}
\caption{Solution due to a point source located at $(1/10, 10)$. {\bf Left}: real part of the incident field. {\bf Middle}: real part of the scattered field. {\bf Right}: absolute value of the total field. In all three cases the black line denotes the real part of the interface, and the blue line is its imaginary part.}
\label{fig:pt_src}
\end{figure}

\begin{figure}[h!]
\centering
\includegraphics[width=\textwidth]{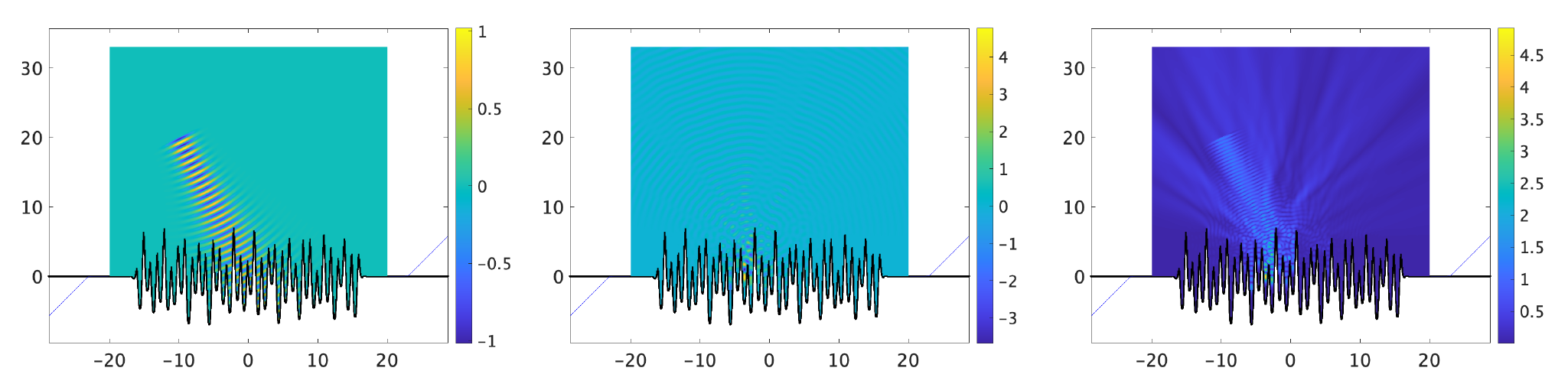}
\caption{Solution due to a Gaussian beam located at $(-10 + 4i, 20 - 8i)$. {\bf Left}: real part of the incident field. {\bf Middle}: real part of the scattered field. {\bf Right}: absolute value of the total field. In all three cases the black line denotes the real part of the interface, and the blue line is its imaginary part.}
\label{fig:gb}
\end{figure}

\subsection{Three dimensional examples}
To illustrate the performance in three dimensions, we consider scattering from a perturbed half-plane parametrized
as
\begin{equation}
\begin{bmatrix}
x_{1}(u,v) \\ 
x_{2}(u,v) \\
x_{3}(u,v) 
\end{bmatrix}
= \begin{bmatrix}
\left( 4 e^{-\frac{(u^2 + v^2)}{10}} + \left(2 - \erfc{(\sqrt{u^2 + v^2} + 6)}\right)/2 \right) u \\
\left( 4 e^{-\frac{(u^2 + v^2)}{10}} + \left(2 - \erfc{(\sqrt{u^2 + v^2} + 6)}\right)/2 \right) v \\
2 e^{-\frac{(u^2 + v^2)}{10}} + 1.5e^{-2(\sqrt{u^2 + v^2} - 8)^2}
\end{bmatrix}
\, , \quad (u,v) \in (-\infty, \infty)^2 \, .
\end{equation}
The truncated complexified contour $\tGammat$ is parameterized as
\begin{equation}
\begin{bmatrix}
x_{1}(u,v) \\ 
x_{2}(u,v) \\
x_{3}(u,v) 
\end{bmatrix}
= \begin{bmatrix}
\left( 4 e^{-\frac{(u^2 + v^2)}{10}} + \left(2 - \erfc{(\sqrt{u^2 + v^2} + 6)}\right)/2 \right) u + \phi_{i,3d}(u) \\
\left( 4 e^{-\frac{(u^2 + v^2)}{10}} + \left(2 - \erfc{(\sqrt{u^2 + v^2} + 6)}\right)/2 \right) v + \phi_{i,3d}(v) \\
2 e^{-\frac{(u^2 + v^2)}{10}} + 1.5e^{-2(\sqrt{u^2 + v^2} - 8)^2}
\end{bmatrix}
\, , \quad (u,v) \in (-30, 30)^2 \, .
\end{equation}
with 
\begin{equation*}
\phi_{i,3d}(t) = \phi(0.75 (t + 10) - \phi(0.75(10 - t)) \, .
\end{equation*}
We plot the geometry and the complexified contours in~\cref{fig:3d_geom}.

\begin{figure}[h!]
\centering
\includegraphics[width=0.45\textwidth]{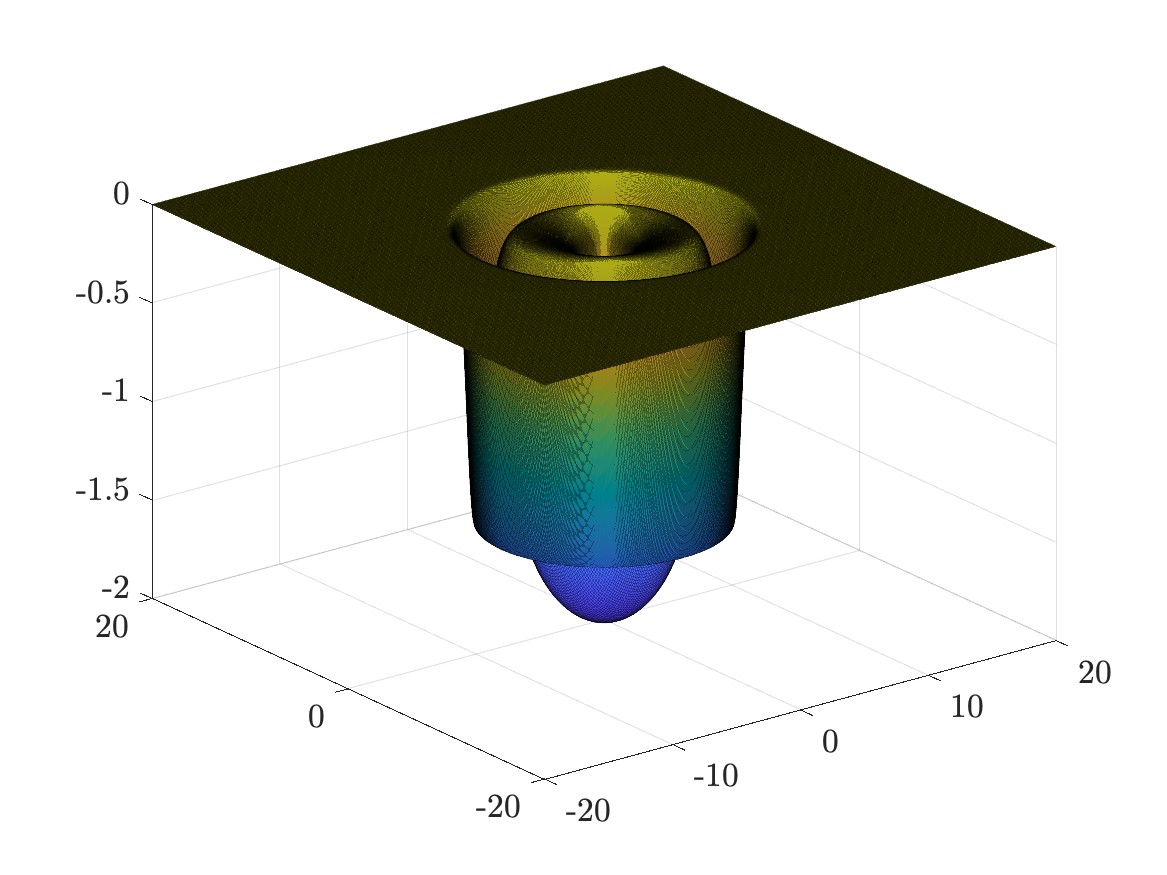}
\includegraphics[width=0.45\textwidth]{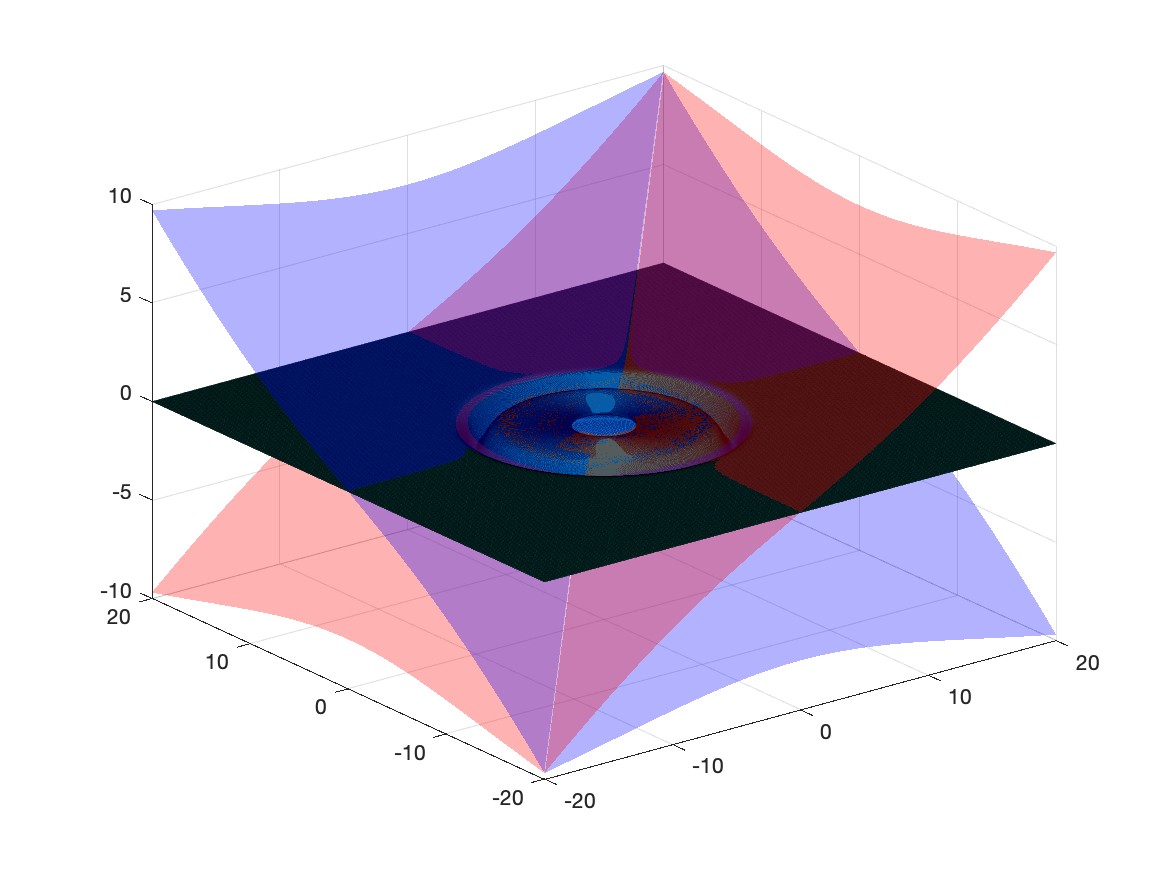}
\caption{Geometry for three-dimensional examples. {\bf Left}: the real surface. The region $\Omega$ is taken to be the portion of $\mathbb{R}^3$ above the surface. {\bf Right}: visualization of complexification. The red surface is the imaginary part of the $x_1$ component (plotted against the real parts of $x_1$ and $x_2$). The blue surface corresponds to the imaginary part of the $x_2$ component (plotted against the real parts of $x_1$ and $x_2$). }
\label{fig:3d_geom}
\end{figure}

In this example, the wavenumber $k = \pi$.
When the boundary is discretized with $1001 \times 1001$ equispaced points in $(u,v)$ using $5$th order quadrature corrections, the accuracy of the computed solution estimated via an analytical solution test is approximately $10^{-4}$. In the left panel of~\cref{fig:3d_pwave} we plot the $\log_{10}$ of the absolute value of the density on the surface due to an incident plane wave with wave-vector $k(1/2, 1/2, 1/\sqrt{2})$ (the reflected wave, as before, is included in the incident field), the middle panel shows the real part of the scattered field and the right panel the absolute value of the total field along the plane $\{x_2=0\}.$

\begin{figure}[h!]
\centering
\includegraphics[width=\textwidth]{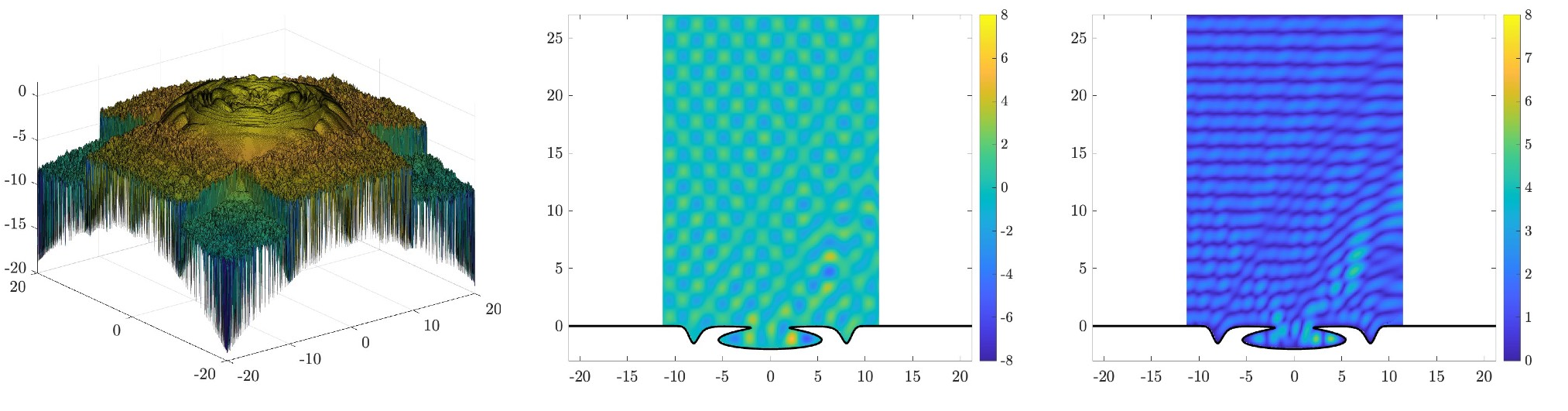}
\caption{Scattering of a plane wave from the geometry of \cref{fig:3d_geom}. The plane wave has unit amplitude and wavevector ${\bf k} = k(1/2,1/2,-1/\sqrt{2})$, with $k = \pi$. {\bf Left}: the $\log_{10}$ absolute value of the density solved for in the complexified BIE. {\bf Middle}: the real part of the computed solution. {\bf Right}: the absolute value of the computed solution.}
\label{fig:3d_pwave}
\end{figure}

\section{Concluding remarks}
In this work, we developed a coordinate complexification approach for the efficient solution of acoustic scattering problem from sound-soft perturbed half-spaces in two and three dimensions. In particular, we prove that the integral equation admits an analytical extension to complex coordinates (in its flat parts), and in appropriate quadrants of the complex plane, the incoming scattering data, and the solution to the integral equation exhibit exponential decay with a rate proportional to the wave number of the problem. The proof relies on the asymptotic behavior of the Green's function, the outgoing nature and analytic extensibility of the data, which together imply the outgoing nature of the solution to the integral equation. 

This observation enables the efficient discretization of the integral equation on the complexified contour, as the integral equation can be truncated at $O(\log{(1/\eps)}/k)$ distance in the infinite flat parts (see~\cref{subsec:trunc-err}), as opposed to $O(1/\eps^{2})$ in order to achieve a precision $\epsilon$. While there exist several other numerical approaches which achieve similar computational complexity, our approach is mathematical rigorous, and numerically easy to implement. We illustrate the accuracy and efficiency of our approach through several numerical examples. 

There are many natural extensions of this work, such as its extension to fast algorithms like the fast multipole methods, and fast direct solvers for problems with complex coordinates. This general approach can be applied to the analysis of dielectric acoustic half-spaces,  perfect electric conductors, as well as impedance and dielectric scattering for electromagnetic applications. In a forthcoming paper we show how this method leads to  accurate and efficient  solvers for layered media with multiple perturbed half-spaces, and open waveguides.

\section{Acknowledgments}
C. Epstein, J. Hoskins, S. Jiang, and M. Rachh would  like to thank the American Institute of Mathematics and, in particular, John Fry for hosting them on Bock Cay during the SQuaREs program, where parts of this work were completed.


\bibliographystyle{siam}
\bibliography{journalnames,fmm}

\end{document}